\documentclass[a4paper,reqno,11pt]{amsart}

\newtheorem{theorem}{Theorem}[section]
\newtheorem{proposition}[theorem]{Proposition}
\newtheorem{lemma}[theorem]{Lemma}
\newtheorem{corollary}[theorem]{Corollary}

\newtheorem*{conjecture*}{Conjecture}
\theoremstyle{definition}
\newtheorem{claim}{Claim}
\newtheorem*{question}{Question}

\newtheorem{subclaim}{Claim}[claim]

\usepackage{amssymb}
\usepackage{hyperref}
\usepackage{enumitem}
\usepackage{mathrsfs}
\usepackage[all,cmtip]{xy}

\allowdisplaybreaks

\newcommand{\Gal}{\mathrm{Gal}}
\newcommand{\Cl}{\mathrm{Cl}}
\newcommand{\Fitt}{\mathrm{Fitt}}
\newcommand{\Frob}{\mathrm{F}}
\newcommand{\Ker}{\mathrm{Ker}}

\begin{document}

\title{Index formulae for Stark units and their solutions}

\author{Xavier-Fran\c{c}ois Roblot}

\address{Tokyo Institute of Technology, Department of Mathematics, 2-12-1 Ookayama, Meguro-ku, Tokyo 
152-8550, Japan}
\email{roblot@math.titech.ac.jp}

\thanks{Supported by the JSPS Global COE \emph{CompView}.}

\date{\today \ (draft v1)}

\begin{abstract}
Let $K/k$ be an abelian extension of number fields with a distinguished place of $k$ that splits totally in $K$. In that situation, the abelian rank one Stark conjecture predicts the existence of a unit in $K$, called the Stark unit, constructed from the values of the $L$-functions attached to the extension. In this paper, assuming the Stark unit exists, we prove index formulae for it. In a second part, we study the solutions of the index formulae and prove that they admit solutions unconditionally for quadratic, quartic and sextic (with some additional conditions) cyclic extensions. As a result we deduce a weak version of the conjecture (``up to absolute values'') in these cases and precise results on when the Stark unit, if it exists, is a square. 
\end{abstract}

\maketitle

\section{Introduction}

Let $K/k$ be an abelian extension of number fields. Denote by $G$ its Galois group. Let $S_\infty$ and $S_\textrm{ram}$ denote respectively the set of infinite places of $k$ and the set of finite places of $k$ ramified in $K/k$. Let $S(K/k) := S_\infty \cup S_\mathrm{ram}$. Fix a finite set $S$ of places of $k$ containing $S(K/k)$ and of cardinality at least $2$. Assume that there exists at least one place in $S$, say $v$, that splits totally in $K/k$ and fix a place $w$ of $K$ dividing $v$. Let $e$ be the order of the group of roots of unity in $K$. In this setting Stark \cite{stark:4} made the following conjecture. 

\begin{conjecture*}[\sc Abelian rank one Stark conjecture]\ \\
 There exists an $S$-unit $\varepsilon_{K/k,S}$ in $K$ such that
\begin{enumerate}[label=\textup{(\arabic{*})}]
\item For all characters $\chi$ of $G$
\begin{equation*}
L'_{K/k,S}(0, \chi) = \frac{1}{e} \sum_{g \in G} \chi(g) \log |\varepsilon_{K/k,S} ^g|_w  
\end{equation*}
where $L_{K/k,S}(s, \chi)$ denotes the $L$-function associated to $\chi$ with Euler factors at prime ideals in $S$ deleted. 
\item The extension $K( \varepsilon_{K/k,S}^{\,1/e})/k$ is abelian. 
\item If furthermore $|S| \geq 3$ then $\varepsilon$ is a unit of $K$.
\end{enumerate}
\end{conjecture*}

The unit $\varepsilon_{K/k,S}$ is called the Stark unit associated to the extension $K/k$, the set of places $S$ and the infinite place $v$.\footnote{In fact the place $w$ but changing the place $w$ just amounts to replace the Stark unit by one of its conjugate.} It is unique up to multiplication by a root of unity in $K$. A good reference for this conjecture is \cite[Chap.~IV]{tate:book}. 

The starting point of this research is the conjectural method used in \cite{cohen-roblot} and \cite{stark12} (and inspired by  \cite{stark:example}) to construct totally real abelian extensions of totally real fields. Let $L/k$ be such an extension. The idea is to construct a quadratic extension $K/L$, abelian over $k$, satisfying some additional conditions similar to the assumptions (A1), (A2) and (A3) below. Assuming the Stark conjecture for $K/k$, $S(K/k)$ and a fixed real place $v$ of $k$, one can prove that $K = k(\varepsilon)$ and $L = k(\alpha)$ where $\alpha := \varepsilon + \varepsilon^{-1}$ and $\varepsilon := \varepsilon_{K/k, S(K/k)}$ is the corresponding Stark unit. Using part~(1) of the conjecture, one computes the minimal polynomial $A(X)$ of $\alpha$ over $k$. The final step is to check unconditionally that the polynomial $A(X)$ does indeed define the extension $L$. 

One notices in that setting that the rank of the units of $K$ is equal to the rank of units of $L$ plus the rank of the module generated by the Stark unit and its conjugates over $k$. A natural question to ask is wether the index of the group generated by the units of $L$ and the conjugates of the Stark unit has finite index inside the group of units of $K$ and, if so, if this index can be computed. A positive answer to the first question is given by Stark in \cite[Th.~1]{stark:3}. In \cite{arakawa}, Arakawa gives a formula for this index when $k$ is a quadratic field. Using similar methods, we obtain a general result (Theorem~\ref{th:globindex}) in the next section. Then we derive a ``relative'' index formula (Theorem~\ref{th:main}) that relates the index of the subgroup generated over $\mathbb{Z}[G]$ by the Stark unit inside the ``minus-part'' of the group of units of $K$ to the cardinality of the ``minus-part'' of the class group of $K$.\footnote{Similar in some way to the index formulae for cyclotomic units, see \cite[Chap.~8]{washington:book}.} In the third section, we use results of Rubin \cite{rubin:index} on some type of Gras conjecture for Stark units to show that the relative index formula implies local relative index formulae (Theorem~\ref{th:localindex}). Starting with the fourth section, we stop assuming the abelian rank one Stark conjecture and study directly the solutions to the index formulae.  In section 4, we look at how much these index formulae characterize the Stark unit (Proposition~\ref{prop:prodform} and Corollary~\ref{cor:unicity}). In the next section, we introduce the algebraic tools that will be needed to prove the existence of solutions in some cases in the following sections. We also reprove in that section the abelian rank one Stark conjecture for quadratic extensions (Theorem~\ref{th:quadratic}). Finally, sections~6 and 7 are devoted to a proof that solutions to the index formulae always exist for quartic extensions (Theorem~\ref{th:quartic}) and sextic extensions (Theorem~\ref{th:sextic}) with some additional conditions in that case. We show that the existence of solutions in those cases imply a weak version of the conjecture where part~(1) is satisfied only up to absolute values.\footnote{Unfortunately, in most cases the values are complex and there does not appear to be any obvious way to remove these absolute values.} We also obtain results on when the Stark unit, if it exists, is a square (Corollary~\ref{cor:whensquare}, Theorem~\ref{th:quadratic}, Corollary~\ref{cor:square4} and Corollary~\ref{cor:square6}). 
 
\section{The index formulae}\label{sec:index}

We assume from now on that the place $v$ is infinite\footnote{For $v$ a finite place, the abelian rank one Stark conjecture is basically equivalent to the Brumer-Stark conjecture, see \cite[\S IV.6]{tate:book}. Recent results of  Greither and Popescu \cite{gp:iwasawa} imply the validity of the Brumer-Stark conjecture away from its $2$-part and under the hypothesis that an appropriate Iwasawa $\mu$-invariant vanishes.} and that $k$ has at least two infinite places. Therefore we can always apply the conjecture for any finite set $S$ containing $S(K/k)$. The cases that we are excluding are $k = \mathbb{Q}$ and $k$ a complex quadratic field. In both cases the  conjecture is proved and the Stark unit is strongly related to cyclotomic units and elliptic units respectively.

\smallskip

Fix a finite set $S$ of places of $k$ containing $S(K/k)$. We make the following additional assumptions. 
\begin{enumerate}[label=(A\arabic{*})]
\item $k$ is totally real and the infinite places of $K$ above $v$ are real, the infinite places of $K$ not above $v$ are complex. 
\item The maximal totally real subfield $K^+$ of $K$ satisfies $[K:K^+] = 2$.
\item All the finite primes in $S$ are either ramified or inert in $K/K^+$.
\end{enumerate}

\smallskip

If $S$ contains more than one place that splits totally in $K/k$ then the conjecture is trivially true with the Stark unit being equal to $1$. Therefore the only non trivial case excluded by (A1) is the case when $k$ has exactly one complex place and $K$ is totally complex. It is likely that most of the methods and results in this paper can be adapted to cover also that case. Assumptions (A2) and (A3) are necessary to ensure that the rank of the group generated by the units of $K^+$ and the conjugate of the Stark unit has finite index inside the group of units of $K$. Without these assumptions, global index formulae for Stark units as they are stated in this article cannot exist although it is still possible to prove index formulae for some $p$-adic characters if one takes also into account Stark units coming from subextensions (see \cite{rubin:index} or Section~3). 

\smallskip

\noindent\textbf{We assume until further notice that the conjecture is true for the extension $K/k$, the set of places $S$ and the distinguished place $v$.}\footnote{Since $v$ is the only real place of $k$ that stays real in $K$, we will usually not specify it.}

\smallskip

Denote by $\varepsilon := \varepsilon_{K/k, S}$ the corresponding Stark unit. From now on, all subfields of $K$ (including $K$ itself) are identified with their image in $\mathbb{R}$ by $w$. We make the Stark unit unique by imposing that $\varepsilon > 0$. It follows that $\varepsilon^g > 0$ for all $g \in G$, see \cite[\S IV.3.7]{tate:book}. One can also prove under these hypothesis, see \cite[Lem.~2.8]{stark12}, that $|S(K/k)| \geq 3$ and therefore $\varepsilon$ is a unit of $K$ by part (3) of the Conjecture, and that $|\varepsilon|_{w'} = 1$ for any place $w'$ of $K$ not above $v$.  

Let $m$ be the degree of $K^+/k$ and $d$ be the degree of $k/\mathbb{Q}$. Thus we have $[K:k] = 2m$ and $[K:\mathbb{Q}] = 2md$. Let $\tau$ denote the non trivial element of $\Gal(K/K^+)$. It is the complex conjugation of the extension $K$ and, by the above remark, we have $\varepsilon^\tau = \varepsilon^{-1}$. Let $G^+$ denote the Galois group of $K^+/k$, thus $G^+ \cong G/\langle \tau \rangle$. It follows from (A1) that the signatures of $K^+$ and $K$ are respectively $(dm,0)$ and $(2m, m(d-1))$. Therefore the rank of $U_{K^+}$ and $U_K$, the group of units of $K^+$ and $K$, are respectively $dm - 1$ and $2m + m(d-1) - 1 = (dm-1)+m$. Let $U_{\rm Stark}$ be the multiplicative $\mathbb{Z}[G]$-module generated by $\pm 1$, $\varepsilon$ and $U_{K^+}$. Let $R := \{\rho_1, \dots, \rho_m\}$ be a fixed set of representatives of $G$ modulo $\langle \tau \rangle$. Set  $\varepsilon_\ell := \rho_\ell^{-1}(\varepsilon)$ for $\ell = 1, \dots, m$. Since $\tau(\varepsilon) = \varepsilon^{-1}$, the group $U_{\rm Stark}$ is generated over $\mathbb{Z}$ by $\{\pm 1, \eta_1, \dots, \eta_{dm-1}$, $\varepsilon_1, \dots, \varepsilon_m\}$ where $\eta_1, \dots, \eta_{dm-1}$ is a system of fundamental units of $K^+$. Let $|\cdot|_j$, $1 \leq j \leq (d+1)m$ denote the infinite normalized absolute values of $K$ ordered in the following way. The  $2m$ real absolute values of $K$, corresponding to the places over $v$, are $|\cdot|_j := |\rho_j(\cdot)|$ and $|\cdot|_{j+m} := |\rho_j\tau(\cdot)|$ for $1 \leq j \leq m$. The complex absolute values, corresponding to the infinite places not above $v$, are $|\cdot|_j$ for $2m+1 \leq j \leq (d+1)m$.  The regulator of $U_{\rm Stark}$ is the absolute value of the determinant of the following matrix\footnote{We discard the last absolute value $|\cdot|_{(d+1)m}$.}
\begin{equation*}
\left(\begin{array}{c|c}
\log |\eta_i|_j  & \log |\varepsilon_\ell|_j
\end{array}\right)_{j,(i,\ell)}
\end{equation*}
where $1 \leq j \leq (d+1)m-1$, $1 \leq i \leq dm-1$ and $1 \leq \ell \leq m$. For $1 \leq j \leq (d+1)m$, let $|\cdot|^+_j$ denote the restriction of the absolute value $|\cdot|_j$ to $K^+$. For $1 \leq j \leq m$, the places corresponding to $|\cdot|_j$ and $|\cdot|_j^+$ are real and $\log |\eta_i|^+_j = \log |\eta_i|_j =  \log |\eta_i|_{j+m}$. For $2m + 1 \leq j \leq (d+1)m$, the places corresponding to $|\cdot|_j$ and $|\cdot|_j^+$ are respectively complex and real, thus $\log |\eta_i|^+_j = 2 \log |\eta_i|_j$. Note also that $|\varepsilon_\ell|_{j+m} = |\varepsilon_\ell|^{-1}_m$ for $1 \leq j \leq m$ and $|\varepsilon_\ell|_j = 1$ for $2m+1 \leq j \leq (d+1)m$. Therefore the matrix is equal to
\begin{equation*}
\renewcommand{\arraystretch}{1.2}
\left(\begin{array}{c|c}
\log |\eta_i|^+_j &  \log |\varepsilon_\ell|_j \\[2pt]
\hline
\log |\eta_i|^+_j & -\log |\varepsilon_\ell|_j \\[2pt]
\hline
2 \log |\eta_i|^+_{j'}  & 0 \\
\end{array}\right)_{(j,j'),(i,\ell)}
\end{equation*}
where $1 \leq j \leq m$, $2m+1 \leq j' \leq (d+1)m-1$, $1 \leq i \leq dm-1$ and $1 \leq \ell \leq m$. Now we add the $j$-th row to the $(m+j)$-th row for $1 \leq j \leq m$ and we obtain finally the following matrix with the same determinant
\begin{equation*}
\renewcommand{\arraystretch}{1.2}
\left(\begin{array}{c|c}
\log |\eta_i|^+_j &  \log |\varepsilon_\ell|_j \\[2pt]
\hline
2\log |\eta_i|^+_j & 0 \\[2pt]
\hline
2 \log |\eta_i|^+_{j'}  & 0
\end{array}\right)_{(j,j'),(i,\ell)}.
\end{equation*}
Therefore the regulator of $U_{\rm Stark}$ is 
\begin{equation}\label{eq:eq1}
\textrm{Reg}(U_{\rm Stark}) = \left|\det(\log |\varepsilon_\ell|_j)_{j,\ell}\det\left(2\log |\eta_i|_{j'}^+\right)_{j',i}\right|
\end{equation}
where $1 \leq \ell, j \leq m$, $1 \leq i \leq dm-1$ and $j'$ runs through the set $\{1, \dots, m, 2m+1, \dots, (d+1)m-1\}$. The absolute values $|\cdot|^+_1, \dots, |\cdot|^+_m, |\cdot|^+_{2m+1}, \dots, |\cdot|^+_{(d+1)m-1}$ are the absolute values corresponding to all the infinite places of $K^+$ but one. Thus the second term is $2^{dm-1} R_{K^+}$. For the first term, we have
\begin{equation*}
|\det(\log |\varepsilon_\ell|_j)_{j,\ell}| = |\det(\log |\varepsilon^{\rho\lambda^{-1}}|)_{\rho,\lambda \in R}|.
\end{equation*}
We say that a character $\chi$ of $G$ is even if $\chi(\tau) = 1$, otherwise $\chi$ is odd and $\chi(\tau) = -1$. The even characters of $G$ are the inflations of characters of $G^+$. We have the following modification of the classical determinant group factorization. 

\begin{lemma}\label{lem:detgroup}
Let $a_g\in \mathbb{C}$, for $g \in G$, be such that $a_{\tau g} = -a_{g}$ for all $g \in G$. Then 
\begin{equation*}
\det(a_{\rho\lambda^{-1}})_{\rho, \lambda \in R} = \prod_{\chi \textup{ odd}} \sum_{\rho \in R}  \chi(\rho) a_\rho.
\end{equation*}
\end{lemma}

\begin{proof}
Let $E$ be the $\mathbb{C}$-vector space of functions $f : G \to \mathbb{C}$ such that $f(\tau g) = -f(g)$ for all $g \in G$. Clearly it has dimension $m$ and admits $(\chi)_{\chi \text{ odd}}$ has a basis. Another basis is given by the functions $(\delta_\rho)_{\rho \in R}$ defined by
\begin{equation*}
\delta_\rho(\rho) = 1,\ \delta_\rho(\tau\rho) = -1 \text{ and } \delta_\rho(g) = 0 \text{ for all } g \in G \text{ with } g \not= \rho, \tau\rho.
\end{equation*}
The group $G$ acts on $E$ by $f^\sigma : g  \mapsto f(g\sigma)$ for $f \in E$ and $\sigma \in G$. In particular, we have $f^\tau = -f$. We extend this action linearly to give $E$ a structure of $\mathbb{C}[G]$-module. Now consider the endomorphism defined by $T := \sum\limits_{g \in G} a_g g$. We have
\begin{align*}
T(\delta_\rho) & = \sum_{\substack{g \in G \\ \rho g^{-1} \in R}} a_g \delta_\rho^g + \sum_{\substack{g \in G \\ \rho g^{-1} \not\in R}} a_g \delta_\rho^g = \sum_{\substack{g \in G \\ \rho g^{-1} \in R}} a_g \delta_{\rho g^{-1}} - \sum_{\substack{g \in G \\ \rho g^{-1} \not\in R}} a_g \delta_{\tau\rho g^{-1}}. \\
\intertext{We write $\lambda = \rho g^{-1}$ in the first sum and $\lambda = \tau\rho g^{-1}$ in the second one. We get}
T(\delta_\rho) & =  \sum_{\lambda \in R} a_{\rho\lambda^{-1}} \delta_{\lambda} - \sum_{\lambda \in R} a_{\tau\rho\lambda^{-1}} \delta_{\lambda} = 2 \sum_{\lambda \in R} a_{\rho\lambda^{-1}} \delta_{\lambda}.
\end{align*}
Therefore the determinant of $T$ is $2^m \det(a_{\rho\lambda^{-1}})_{\rho, \lambda \in R}$. On the other hand, for $\chi$ odd, we compute 
\begin{align*}
T(\chi) = \sum_{g \in G} a_g \chi^g = \sum_{g \in G} a_g \chi(g) \chi.
\end{align*}
Thus $\chi$ is an eigenvector for $T$ with eigenvalue $\sum\limits_{g \in G} a_g \chi(g) = 2\sum\limits_{\rho \in R}  \chi(\rho) a_\rho$. Therefore $\det(T) = 2^m \prod\limits_{\chi \text{ odd}} \sum\limits_{\rho \in R} \chi(\rho) a_\rho$ and the result follows.
\end{proof}

By the lemma, we get
\begin{align}\label{eq:eq2}
\det(\log |\varepsilon^{\rho\lambda^{-1}}|)_{\rho,\lambda \in R} & = \prod_{\chi \text{ odd}} \sum_{\rho \in R} \chi(\rho)  \log |\varepsilon^\rho| = \frac{1}{2} \prod_{\chi \text{ odd}} \sum_{g \in G} \chi(g)  \log |\varepsilon^g| \notag \\
& = \prod_{\chi \text{ odd}} L'_{K/k,S}(0, \chi)
\end{align}
using part~(1) for the last equality and the fact that the number of roots of unity in $K$ is  $2$ since $K$ is not totally complex by (A1). On the other hand, we have
\begin{equation}\label{eq:deczet}
\prod_{\chi \text{ odd}} L_{K/k,S}(s, \chi)  = \frac{\zeta_{S,K}(s)}{\zeta_{S,K^+}(s)} 
\end{equation}
where $\zeta_{S,K}(s) := \zeta_{S_K, K}(s)$ and $\zeta_{S,K^+}(s) := \zeta_{S_{K^+}, K^+}(s)$ denote respectively the Dedekind zeta functions of $K$ and $K^+$ with the Euler factors at primes in $S_K$ and $S_{K^+}$ removed. Here $S_K$ and $S_{K^+}$ denote respectively the set of places of $K$ and of $K^+$ above the places in $S$. We will often use by abuse the subscript $S$ instead of $S_K$ or $S_{K^+}$ to simplify the notation. Taking the limit when $s \to 0$ in \eqref{eq:deczet} and using the expression for the Taylor development at $s = 0$ of Dedekind zeta functions, see \cite[Cor.~I.1.2]{tate:book}, we get
\begin{equation}\label{eq:eq3}
\prod_{\chi \text{ odd}} L'_{K/k,S}(0, \chi) = 2^{t_S} \frac{h_K R_K}{h_K^+ R_K^+}
\end{equation}
where $t_S$ is the number of prime ideals in $S_{K^+}$ that are inert in $K/K^+$ and $h_K$, $R_K$, $h_{K^+}$ and $R_{K^+}$ are respectively the class numbers and regulators of $K$ and $K^+$. Putting together equations \eqref{eq:eq1}, \eqref{eq:eq2} and \eqref{eq:eq3}, we get the following result. 
\begin{theorem}\label{th:globindex}
The index of $U_{\rm Stark}$ in the group of units of $K$ is 
\begin{equation*}
(U_K: U_{\rm Stark}) = 2^{t_S+dm-1} \frac{h_K}{h_{K^+}}
\end{equation*} 
where $t_S$ is the number of prime ideals in $S_{K^+}$ that are inert in $K/K^+$. \qed
\end{theorem}

Let $\Cl_K$ and $\Cl_{K^+}$ denote respectively the class groups of $K$ and $K^+$. Define $\Cl^-_K$ and $U^-_K$ as the kernel of the following maps induced by the norm $\mathcal{N} := 1 + \tau$  of the extension $K/K^+$
\begin{equation*}
\Cl_K^- := \Ker(\mathcal{N} : \Cl_K \to \Cl_{K^+})
\quad\text{and}\quad
U_K^- := \Ker(\mathcal{N} : \bar{U}_K \to \bar{U}_{K^+})
\end{equation*}
where $\bar{U}_K$ and $\bar{U}_{K^+}$ are respectively $U_K/\{\pm 1\}$ and $U_{K^+}/\{\pm 1\}$. From now on, we use the additive notation to denote the action of $\mathbb{Z}[G]$, and other group rings, on $\bar{U}_K$ and its subgroups $U_K^-, \bar{U}_{K^+}, \dots$. For $x \in U_K$, we denote by $\bar{x}$ its class in $\bar{U}_K$ and adopt the following convention: if $\bar{x} \in \bar{U}_K$, we let $x$ denote the unique element in the class $\bar{x}$ such that $x > 0$. Note that $\mathcal{N}(x) = \mathcal{N}(-x) = 1$ since $K/K^+$ is ramified at at least one real place.

\begin{theorem}\label{th:main}
We have 
\begin{equation*}
\big(U^-_K : \mathbb{Z}[G] \cdot \bar\varepsilon\big) = 2^{e+t_S} |\Cl_K^-|
\end{equation*}
where $2^e = (\bar{U}_{K^+}:\mathcal{N}(\bar{U}_K))$. 
\end{theorem}

\begin{proof}
By class field theory the map $\mathcal{N} : \Cl_K \to \Cl_{K^+}$ is surjective. Therefore $|\Cl_K^-| = h_K/h_{K^+}$. On the other hand, if we let $\bar{U}_{\rm Stark} := U_{\rm Stark}/\{\pm 1\}$, we have 
\begin{equation*}
\Ker \left(\mathcal{N} : \bar{U}_\mathrm{Stark} \to \bar{U}_{K^+}\right) = \mathbb{Z}[G] \cdot \bar\varepsilon
\quad\text{and}\quad
\mathrm{Im} \left(\mathcal{N} : \bar{U}_\mathrm{Stark} \to \bar{U}_{K^+}\right) = 2 \cdot \bar{U}_{K^+}.
\end{equation*}
Therefore we get
\begin{equation*}
(\bar{U}_K: \bar{U}_{\rm Stark}) = (\mathcal{N}(\bar{U}_K) :  2 \cdot \bar{U}_{K^+}) \, \big(U_K^- : \mathbb{Z}[G] \cdot \bar\varepsilon\big).
\end{equation*}
Since $(\bar{U}_K: \bar{U}_{\rm Stark}) = (U_K: U_{\rm Stark})$, it follows from Theorem~\ref{th:globindex} that
\begin{equation*}
\big(U_K^- : \mathbb{Z}[G] \cdot \bar\varepsilon\big) = \frac{2^{t_S+dm-1} |\Cl_K^-|}{\big(\mathcal{N}(\bar{U}_K) : 2 \cdot \bar{U}_{K^+}\big)}.
\end{equation*}
We conclude by noting that
\begin{equation*}
\big(\mathcal{N}(\bar{U}_K) : 2 \cdot \bar{U}_{K^+}\big) = \frac{\big(\bar{U}_{K^+} : 2 \cdot \bar{U}_{K^+}\big)}{\big(\bar{U}_{K^+}: \mathcal{N}(\bar{U}_K) \big)} = \frac{2^{dm-1}}{\big(\bar{U}_{K^+}: \mathcal{N}(\bar{U}_K) \big)}.\qedhere
\end{equation*}
\end{proof}

It has been observed that the Stark unit is quite often a square. The theorem provides us with a necessary condition for that to happen.

\begin{corollary}\label{cor:whensquare}
Let $c$ be the $2$-valuation of the order of $\Cl_K^-$. A necessary condition for the Stark unit $\varepsilon$ to be a square in $K$ is  
\begin{equation*}
e+t_S+c \geq m. 
\end{equation*}
\end{corollary}

\begin{proof}
Assume that $\varepsilon = \eta^2$ with $\eta \in K$. Then it is easy to see that $\eta \in U_K^-$ and therefore 
$\big(\mathbb{Z}[G] \cdot \bar\varepsilon : \mathbb{Z}[G] \cdot \bar\eta\big) = 2^m$ divides $2^{e+t_S} |\Cl_K^-|$.
\end{proof}

We will see below, see \eqref{eq:bounde}, that $e \geq (d-1)m-2$. Therefore the inequality in the theorem is always satisfied for $d \geq 2 + 2/m$. However, this is not enough to ensure that the Stark unit is a square in general. Indeed at the end of the paper we give an example of a cyclic sextic extension $K/k$ satisfying (A1), (A2) and (A3), and with $k$ a totally real cubic field where the Stark unit, assuming it exists, is not a square even though $e > m$. But, in all the cases that we study, we can prove that for $d$ sufficiently large the Stark unit is always a square. Of course theses cases are quite specific and it is difficult to draw from them general conclusions, but still we are lead to ask the following question. 

\begin{question}
Fix a relative degree $m$. Does there exist a constant $D(m)$, depending only on $m$, such that for any extensions $K/k$ of degree $2m$ and any finite set of places $S$ containing $S(K/k)$ satisfying (A1), (A2) and (A3), and with $d \geq D$, the corresponding Stark unit, assuming that it exists, is always a square in $K$?
\end{question}

\section{Rubin's index formula}

In \cite{rubin:index}, Rubin proves Gras conjecture type results for Stark units using Euler systems. His results are generalized by Popescu \cite{popescu:survey}. In this section, we use the results of Rubin to get a similar result in our settings. To be able to use Rubin's results we need to make the following additional assumption:
\begin{enumerate}[label=(A\arabic{*})]
 \setcounter{enumi}{3}
\item $K$ contains the Hilbert Class Field $H_k$ of $k$.
\end{enumerate}

\smallskip

\noindent\textbf{We assume in this section that the conjecture is true for the extensions and set of places as described in \cite{rubin:index}.}

\smallskip

We first introduce the results of Rubin. Let $\mathfrak{f}$ be the conductor of $K/k$. For any modulus $\mathfrak{g}$ dividing $\mathfrak{f}$, let $K_\mathfrak{g} = K \cap k(\mathfrak{g})$ be the intersection of $K$ with the ray class field of $k$ of conductor $\mathfrak{g}$. Since $v$ is totally split in $K/k$, one can apply the conjecture to the extension $K_\mathfrak{g}/k$, the set of places $S(K_\mathfrak{g}/k)$ and the place $v$, and get a Stark unit that we denote by $\varepsilon_\mathfrak{g}$. Let $G_\mathfrak{g}$ be the Galois group of $K_\mathfrak{g}/k$. Note that by (A1) the group of roots of unity in $K_\mathfrak{g}$ is $\{\pm 1\}$. Part~(2) of the conjecture is equivalent to the fact that $\varepsilon_\mathfrak{g}^{g - 1} \in U_{K_\mathfrak{g}}^2$ for all $g \in G_\mathfrak{g}$, see \cite[Prop.~IV.1.2]{tate:book}. Define $R_{\rm Stark}$ as the following $\mathbb{Z}[G]$-module
\begin{equation*}
R_{\rm Stark} = \langle \pm 1,\ (\varepsilon_\mathfrak{g}^{g - 1})^{1/2} \text{ for } \mathfrak{g} \mid \mathfrak{f} \text{ and } g \in G_\mathfrak{g} \rangle_{\mathbb{Z}[G]}.
\end{equation*}
Let $p$ be a prime number that does not divide the order of $G$. In particular, $p$ is an odd prime. Denote by $\hat{G}_p$ the set of irreducible $\mathbb{Z}_p$-characters of $G$. For $\psi \in \hat{G}_p$ and $M$ a $\mathbb{Z}[G]$-module, we set
\begin{equation*}
M^\psi := M \otimes_{\mathbb{Z}[G]} \mathbb{Z}_p[\psi]
\end{equation*}
where $\mathbb{Z}_p[\psi]$ is the ring generated over $\mathbb{Z}_p$ by the values of $\psi$ and $G$ acts on $\mathbb{Z}_p[\psi]$ via the character $\psi$. The following result is a direct consequence of Theorem~4.6 of \cite{rubin:index}. 

\begin{theorem}[\sc Rubin]
If $\psi \in \hat{G}_p$ is odd then 
\begin{equation*}
\left|(U_K/R_{\rm Stark})^\psi\right| = \left|\Cl_K^\psi\right|.
\end{equation*}
\end{theorem}

From this we deduce an analogue statement for our case.
\begin{theorem}\label{th:localindex}
For all $\psi \in \hat{G}_p$, we have
\begin{equation*}
\left|(U_K^-/\mathbb{Z}[G] \cdot \bar\varepsilon)^\psi\right| = \left|(\Cl_K^-)^\psi\right|. 
\end{equation*}
\end{theorem}

\begin{proof}
For $M$ a $\mathbb{Z}[G]$-module and $\psi \in \hat{G}_p$, it is direct to see that $M^\psi = (M^{1+\tau})^\psi$ if $\psi$ is even and $M^\psi = (M^{1-\tau})^\psi$ if $\psi$ is odd. In particular, if $\psi$ is even, we get $|(U_K^-/\mathbb{Z}[G] \cdot \bar\varepsilon)^\psi| = |(\Cl_K^-)^\psi| = 1$ and the result follows trivially in that case. Assume now that $\psi$ is odd. Let $\varepsilon_0$ be the Stark unit corresponding to the extension $K/k$, the set of places $S(K/k)$ and the distinguished place $v$. Assume first that $S = S(K/k) \cup \{\mathfrak{p}\}$ for some finite prime ideal $\mathfrak{p}$ of $k$ not in $S(K/k)$. It follows from \cite[Prop.~IV.3.4]{tate:book} that $\bar\varepsilon = (1-\Frob_\mathfrak{p}(K/k)) \cdot \bar\varepsilon_0$ where $\Frob_\mathfrak{p}(K/k)$ is the Frobenius at $\mathfrak{p}$ for the extension $K/k$. By (A3), $\tau$ is a power of $\Frob_\mathfrak{p}(K/k)$ and thus $\psi(\Frob_\mathfrak{p}(K/k))$ is a non trivial root of unity of order dividing $|G|$. Then $1 - \psi(F_\mathfrak{p}(K/k))$ is a $p$-adic unit and therefore $(\mathbb{Z}[G] \cdot \bar\varepsilon)^\psi = (\mathbb{Z}[G] \cdot \bar\varepsilon_0)^\psi$. By repeating this argument if necessary, we see that this last equality also holds in the general case. Now, by taking $\mathfrak{g} = \mathfrak{f}$ and $\sigma = \tau$ in the definition of $R_{\rm Stark}$, we see that $\varepsilon_0^{(\tau-1)/2} = \varepsilon_0^{-1} \in R_{\rm Stark}$. Therefore we have $\varepsilon_0^{\mathbb{Z}[G]} \subset R_{\rm Stark} \subset U_K$, and thus
\begin{equation*}
\varepsilon_0^{2\mathbb{Z}[G]} \subset R_{\rm Stark}^{\tau - 1} \subset U_K^{\tau - 1}. 
\end{equation*}
We take the $\psi$-component, by the above remarks and the theorem, we get
\begin{multline*}
|(U_K^-/\mathbb{Z}[G] \cdot \bar\varepsilon)^\psi| = |(U_K^-/\mathbb{Z}[G] \cdot \bar\varepsilon_0)^\psi| = |(U_K^{\tau - 1}/\varepsilon_0^{2\mathbb{Z}[G]})^\psi| \geq \\
|(U_K^{\tau-1}/R_{\rm Stark}^{\tau-1})^\psi| = |(U_K/R_{\rm Stark})^\psi|
= |\Cl_K^\psi| = |(\Cl_K^-)^\psi|.
\end{multline*}
Assume there exists a character $\psi$ for which this is a strict inequality. Multiplying over all characters in $\hat{G}_p$, we get $|(U_K^-/\mathbb{Z}[G] \cdot \bar\varepsilon) \otimes \mathbb{Z}_p| > |\Cl_K^- \otimes \mathbb{Z}_p|$, a contradiction with Theorem~\ref{th:main}. Therefore the equality holds for all $\psi \in \hat{G}_p$ and the theorem is proved.
\end{proof}

\section{The index property}

\noindent\textbf{From now on, we do not assume any more that the conjecture is true.}

\smallskip

From the results of the previous sections, we see that the conjecture implies that there exists a unit $\bar\varepsilon \in U_K^-$ such that\footnote{Although assumptions (A1) to (A4) are necessary to prove that the Stark unit is a solution of (P2), it is not necessary to assume (A4) to prove that solutions exist in the cases that we study below. It is an interesting question whether or not one could prove that the Stark unit is a solution to (P2) without having first to assume (A4).} 
\begin{enumerate}[label=\textup{(P\arabic{*})}]
\item $\left(U_K^- : \mathbb{Z}[G] \cdot \bar\varepsilon\right) = 2^{e+t_S} |\Cl_K^-|$,
\item $\left|(U_K^-/\mathbb{Z}[G] \cdot \bar\varepsilon)^\psi\right| = \left|(\Cl_K^-)^\psi\right|$ for all $p \nmid [K:k]$ and $\psi \in \hat{G}_p$.
\end{enumerate}

\smallskip

A priori the existence of a solution to (P1) and (P2) does not imply in return the conjecture (except for quadratic extensions, see Theorem~\ref{th:quadratic} below). Indeed, in general, properties (P1) and (P2) do not even  characterize the Stark unit $\varepsilon$. To see that assume that $\bar\eta$ is a solution to (P1) and (P2), and let $\bar\eta' := u \cdot \bar\eta$ where $u \in \mathbb{Z}[G]^\times$ is a unit of $\mathbb{Z}[G]$. Then $\bar\eta'$ also satisfies (P1) and (P2). If $u$ belongs to $\{\pm \gamma : \gamma \in G\} \subset \mathbb{Z}[G]^\times$, the group of trivial units of $\mathbb{Z}[G]$, then $\bar\eta'$ is essentially the \textit{same solution} since it is a conjugate of $\bar\eta$ or the inverse of a conjugate of $\bar\eta$. However there may be some non trivial units in $\mathbb{Z}[G]$ (see the end of this section) and thus solutions to (P1) and (P2) that are not related in any obvious way to the Stark unit. In any case, we have the following result that shows that solutions to (P1) satisfy a very weak version of part~(1) of the conjecture. 

\begin{proposition}\label{prop:prodform}
Let $\bar\eta$ be an element of $U_K^-$ satisfying \textup{(P1)}. Then we have
\begin{equation}\label{eq:factform}
\prod_{\chi\ \mathrm{odd}} \frac{1}{2} \sum_{g \in G} \chi(g) \log |\eta^g| = \pm \prod_{\chi\ \mathrm{odd}} L'_{K/k,S}(0, \chi).
\end{equation}
\end{proposition}

\begin{proof}
Let $\bar x \in U_K^-$. Using the notations of Section~\ref{sec:index}, we have $|x^\tau|_j = |x|_j$ for $2m+1 \leq j \leq (d+1)m$ since these absolute values are complex and $\tau$ is the complex conjugation. Since, by construction, we have $x^\tau = x^{-1}$, it follows that $|x|^2_j = |x^{1+\tau}|_j = 1$ and $|x|_j = 1$ for $2m+1 \leq j \leq (d+1)m$. We can therefore reproduce the determinant computation done in Section~\ref{sec:index} replacing $\varepsilon$ by $\eta$ and $U_\mathrm{Stark}$ by the subgroup $U_0$ of $U_K$ generated by $U_{K^+}$ and the conjugates of $\eta$. We get
\begin{equation*}
(U_K:U_0) = \pm 2^{dm-1} \frac{R_{K^+}}{R_K} \prod_{\chi\ \mathrm{odd}} \frac{1}{2} \sum_{g \in G} \chi(g) \log |\eta^g|.
\end{equation*}
We then proceed as in Theorem~\ref{th:main} by looking at the kernel of the norm map acting on $U_0/\{\pm 1\}$. Since $\bar\eta$ satisfies (P1), it follows that
\begin{equation*}
2^{dm-1} \frac{R_{K^+}}{R_K} \prod_{\chi\ \mathrm{odd}} \frac{1}{2} \sum_{g \in G} \chi(g) \log |\eta^g| = \pm 2^{dm-1+t_S} |\Cl_K^-|. 
\end{equation*}
Then, by \eqref{eq:eq3}, we get the result
\begin{equation*}
\prod_{\chi\ \mathrm{odd}} \frac{1}{2} \sum_{g \in G} \chi(g) \log |\eta^g| = \pm 2^{t_S} \frac{h_K R_K}{h_{K^+} R_{K^+}} = \pm \prod_{\chi\ \mathrm{odd}} L'_{K/k,S}(0, \chi). \qedhere
\end{equation*}
\end{proof}

We now turn to the study of the structure of the $\mathbb{Q}[G]$-module $U_K^- \otimes \mathbb{Q}$. Since $U_K^-$ is killed by $1+\tau$, it is a $\mathbb{Q}[G]^-$-module where $\mathbb{Q}[G]^- := e^- \mathbb{Q}[G]$ and $e^- := \frac{1}{2}(1 - \tau)$ is the sum of the idempotents of odd characters of $G$.\footnote{Note that $\mathbb{Q}[G]^-$ is a ring with identity $e^-$.} Since $U_K^-$ injects into $U_K^- \otimes \mathbb{Q}$, we will identify it with its image. The following result describes the structure of $U_K^- \otimes \mathbb{Q}$ as a Galois module. 

\begin{proposition}\label{prop:iso}
The module $U_K^- \otimes \mathbb{Q}$ is a free $\mathbb{Q}[G]^-$-module of rank $1$.
\end{proposition}

\begin{proof}
Let $\mathcal{Y}_K$ be the $\mathbb{Q}$-vector space with basis the elements $z$ in the set $S_\infty(K)$ of infinite places of $K$. The group $G$ acts on $\mathcal{Y}_K$ in the following way: $z^g$ for $g \in G$ and $z \in S_\infty(K)$ is the infinite place defined by $x \mapsto z(x^g)$ for all $x \in K$. Denote by $\mathcal{X}_K$ the subspace of elements $\sum_z a_z \, z \in \mathcal{Y}_K$ such that $\sum_z a_z = 0$. Then the two $\mathbb{Q}[G]$-modules $\mathcal{X}_K$ and $U_K\otimes \mathbb{Q}$ are isomorphic by a result of Herbrand and Artin \cite{artin:units}. Fix an isomorphism $f : U_K \otimes \mathbb{Q} \to \mathcal{X}_K$. A direct computation shows that $\mathcal{X}_K^- := f(U_K^- \otimes \mathbb{Q})$ is spanned by the vectors $\{w^\rho - w^{\rho\tau}\}_{\rho \in R}$ where $w$ is the fixed place of $K$ above $v$. In particular, $\mathcal{X}_K^-$ is generated as a $\mathbb{Q}[G]^-$-module  by the vector $w - w^\tau$. This proves the result.
\end{proof}
 
\begin{corollary}
There exist $\bar\theta \in \bar{U}_K^-$ and $q \in \mathbb{Q}^\times$ such that 
\begin{equation*}
\prod_{\chi\ \mathrm{odd}} \frac{1}{2} \sum_{g \in G} \chi(g) \log |\theta^g| = q \prod_{\chi\ \mathrm{odd}} L'_{K/k,S}(0, \chi). 
\end{equation*}
\end{corollary}

\begin{proof}
From the proposition, there exists $u \in U_K^- \otimes \mathbb{Q}$ such that $U_K^- \otimes \mathbb{Q} = \mathbb{Q}[G]^- \cdot u$. We let $\bar\theta := n \cdot u$ where $n \in \mathbb{N}$ is large enough so that $\bar\theta \in U_K^-$. Then we set
\begin{equation*}
q := \frac{(U_K^- : \mathbb{Z}[G] \cdot \bar\theta)}{2^{e+t} |\Cl_K^-|}. 
\end{equation*}
The result follows by the proof of Proposition~\ref{prop:prodform} \textit{mutatis mutandis} and replacing $q$ by $-q$ if necessary.
\end{proof}

Thanks to Proposition~\ref{prop:iso}, it is enough to study the structure of $\mathbb{Q}[G]^-$ to understand that of $U_K^- \otimes \mathbb{Q}$. Let $X$ be the set of irreducible $\mathbb{Z}$-characters of $G$. Each $\xi \in X$ is the sum of the irreducible characters in a conjugacy class $C_\xi$ of $\hat{G}$ under the action of $\Gal(\bar{\mathbb{Q}}/\mathbb{Q})$. For $\xi \in X$, we let $e_\xi := \sum_{\chi \in C_\xi} e_\chi \in \mathbb{Q}[G]$ be the corresponding rational idempotent where $e_\chi$ denotes the idempotent associated to the character $\chi$. We have
\begin{equation*}
\mathbb{Q}[G] = \bigoplus_{\xi \in X} e_\xi \mathbb{Q}[G] \simeq \bigoplus_{\xi \in X} \mathbb{Q}(\xi)
\end{equation*}
where $\mathbb{Q}(\xi)$ is the cyclotomic field generated by the values of any character in $C_\xi$. Let $X_\mathrm{odd}$ be the set of $\mathbb{Z}$-characters $\xi \in X$ such that one, and thus all, characters in $C_\xi$ are odd. We have $e^- = \sum_{\xi \in X_\mathrm{odd}} e_\xi$ and from the above decomposition, we get
\begin{equation}\label{eq:QG-}
\mathbb{Q}[G]^- = \bigoplus_{\xi \in X_\mathrm{odd}} e_\xi \mathbb{Q}[G] \simeq \bigoplus_{\xi \in X_\mathrm{odd}} \mathbb{Q}(\xi).
\end{equation}
We now define $\mathbb{Z}[G]^ - := e^- \mathbb{Z}[G]$ and let $\mathcal{O}_G^-$ be the maximal order of $\mathbb{Q}[G]^-$. We have
\begin{equation}\label{eq:OG-}
\mathcal{O}_G^- = \bigoplus_{\xi \in X_\mathrm{odd}} e_\xi \mathbb{Z}[G] \simeq \bigoplus_{\xi \in X_\mathrm{odd}} \mathbb{Z}[\xi].
\end{equation}
Now let $p$ be a prime number. By \eqref{eq:QG-}, we get 
\begin{equation}\label{eq:QpG-}
\mathbb{Q}_p[G]^- \simeq \bigoplus_{\xi \in X_\mathrm{odd}} \mathbb{Q}(\xi) \otimes_\mathbb{Q} \mathbb{Q}_p \simeq  \bigoplus_{\xi \in X_\mathrm{odd}} \bigoplus_{\mathfrak{p} \in S_{\xi, p}} \mathbb{Q}(\xi)_\mathfrak{p}
\end{equation}
where $S_{\xi, p}$ is the set of prime ideals of $\mathbb{Q}(\xi)$ above $p$ and $\mathbb{Q}(\xi)_\mathfrak{p}$ is the completion of $\mathbb{Q}(\xi)$ at the prime ideal $\mathfrak{p}$. On the other hand, each rational character $\xi \in X$ is the sum of irreducible $\mathbb{Z}_p$-characters, say $\xi = \sum_{\psi \in C_{\xi, p}} \psi$, and we have 
\begin{equation*}
\mathbb{Q}_p[G]^- = \bigoplus_{\xi \in X_\mathrm{odd}} \bigoplus_{\psi \in C_{\xi, p}} e_\psi \mathbb{Q}_p[G]^-.
\end{equation*}
Therefore there is a bijection between the prime ideals in $S_{\xi, p}$ and the characters in $C_{\xi, p}$. For $\mathfrak{p}$ a prime ideal in $S_{\xi, p}$, we denote by $\psi_{\xi,\mathfrak{p}}$ the corresponding irreducible $\mathbb{Z}_p$-character. Before stating the first result, we need one more notation. Let $T$ be a set of primes. We say that an element $u \in \mathbb{Q}[G]$ is a $T$-unit if $u \in \mathbb{Z}_p[G]^{-,\times}$ for all $p \not\in T$ where $ \mathbb{Z}_p[G]^{-,\times}$ is the group of units of $\mathbb{Z}_p[G]^-$. 

\begin{proposition}\label{prop:unicity}
Let $M$ be a sub-$\mathbb{Z}[G]^-$-module of $\mathbb{Q}[G]^-$ of finite index. Let $x$ be an element of $M$ such that $x\mathbb{Z}[G]^-$ has finite index inside $M$. Assume that $y$ is another element of $M$ such that
\begin{equation*}
(M: x\mathbb{Z}[G]^-) = (M: y\mathbb{Z}[G]^-) 
\quad\text{and}\quad
(M^\psi: (x\mathbb{Z}_p[G]^-)^\psi) = (M^\psi : (y\mathbb{Z}_p[G]^-)^\psi)
\end{equation*}
for all $p \nmid |G|$ and all $\psi \in \hat{G}_p$ with $\psi$ odd. Then there exists a unique $B$-unit $u \in \mathbb{Q}[G]^-$ such that $y = ux$ where $B$ is the set of primes  dividing both $|G|$ and $(M: x\mathbb{Z}[G]^-)$.
\end{proposition}

\begin{proof}
Since $\mathbb{Q}[G]^- =  x\mathbb{Q}[G]^-$, there exists $u \in \mathbb{Q}[G]^-$ such that $y = ux$. Assume $y = vx$ for another $v \in \mathbb{Q}[G]^-$. Then, for all $\xi \in X_\mathrm{odd}$, we have $\xi(u)\xi(x) = \xi(v)\xi(x)$. Since $\xi(x) \not= 0$, it follows that $\xi(u) = \xi(v)$ and thus by \eqref{eq:QG-}, we get $u = v$ which proves that $u$ is unique. 

Let $p$ be a prime. Assume first that $p$ does not divide $|G|$. Let $\xi \in X_\mathrm{odd}$ and $\mathfrak{p} \in S_{\xi, p}$. Write $\psi := \psi_{\xi, \mathfrak{p}}$ and denote by $\mathbb{Z}[\xi]_{\mathfrak{p}} := \psi(\mathbb{Z}_p[G]^-)$ the ring of integers of $\mathbb{Q}(\xi)_{\mathfrak{p}}$. Then $M^\psi$ is an ideal of $\mathbb{Z}[\xi]_{\mathfrak{p}}$ and we have
\begin{equation*}
\frac{(M^\psi: (x\mathbb{Z}_p[G]^-)^\psi)}{(M^\psi : (y\mathbb{Z}_p[G]^-)^\psi)} = \frac{(M^\psi : \psi(y) \mathbb{Z}[\xi]_{\mathfrak{p}})}{(M^\psi : \psi(x) \mathbb{Z}[\xi]_{\mathfrak{p}})} = |\psi(u)|_{\mathfrak{p}}.
\end{equation*}
Thus $\psi_{\xi, \mathfrak{p}}(u)$ is a unit in $\mathbb{Z}[\xi]_{\mathfrak{p}}$ for all $\xi \in X_\mathrm{odd}$ and $\mathfrak{p} \in S_{\xi, p}$ and thus $u$ lies in $\mathbb{Z}_p[G]^{-,\times}$. Assume now that $p$ does not  divide the index $(M: x\mathbb{Z}[G]^-)$. We have
\begin{equation*}
(M \otimes \mathbb{Z}_p: x\mathbb{Z}_p[G]^-) = (M \otimes \mathbb{Z}_p: y\mathbb{Z}_p[G]^-) = 1.
\end{equation*}
Therefore $x\mathbb{Z}_p[G]^- = M \otimes \mathbb{Z}_p = y\mathbb{Z}_p[G]^-$ and $u \in \mathbb{Z}_p[G]^{-,\times}$.
\end{proof}

By Propositions \ref{prop:iso} and \ref{prop:unicity}, we get the following result.
\begin{corollary}\label{cor:unicity}
Let $B$ be the set of primes that divide both $|G|$ and $|\Cl_K^-|$. Assume there exist $\bar\eta$ and $\bar\eta'$ two elements of $U_K^-$ satisfying \textup{(P1)} and \textup{(P2)}. Then there exists a unique $B$-unit $u \in \mathbb{Q}[G]^-$ such that $\bar\eta' = u \cdot \bar\eta$.  \qed
\end{corollary}

From this result and the discussion at the beginning of the section, one cannot expect the properties (P1) and (P2) to characterize the Stark unit if $\mathbb{Z}[G]^-$ has some non trivial $B$-units and a fortiori if $\mathbb{Z}[G]^-$ has some non trivial units.\footnote{If $\bar\eta$ is a solution to (P1) and (P2) and $u$ is a $B$-unit then $u \cdot \bar\eta$ is not necessarily a solution to (P1) and (P2). A necessary and sufficient condition for that is that the linear map $x \mapsto ux$ of $\mathbb{Q}[G]^-$ has determinant $\pm 1$. This is always true if $u$ is a unit of $\mathbb{Z}[G]^-$.} It follows from the methods of \cite{higman} that $\mathbb{Z}[G]^-$ has some non trivial units if and only if $\mathcal{O}_G^-$ does. By \eqref{eq:OG-}, this is the case if and only if there exists an odd character of $G$ divides $6$. In particular, for $G$ a cyclic group, $\mathbb{Z}[G]^-$ has only trivial units if and only if the order of $G$ is at most $6$. We will prove in the next sections that there exist solutions to (P1) and (P2) in these cases (with some additional conditions for sextic extensions). From this we will deduce another proof of the conjecture for quadratic extensions and a weak version of the conjecture for quartic and sextic extensions.

\section{Algebraic tools} 

In this section we introduce some algebraic tools and results that will be useful in the next sections. We start with the properties of Fitting ideals. Let $R$ be a commutative ring with an identity element. Let $M$ be a finitely generated $R$-module. Therefore there exists a surjective homomorphism $f : R^a \to M$ for some $a \geq 1$. The Fitting ideal of $M$ as an $R$-module, denoted $\Fitt_R(M)$, is the ideal of $R$ generated by $\det(\vec{v}_1, \dots, \vec{v}_a)$ where $\vec{v}_1, \dots, \vec{v}_a$ run through the elements of the kernel of $f$. One can prove that it does not depend on the choice of $f$. We will use the following properties of Fitting ideals, see \cite[Chap.~3]{northcott} or \cite[Chap.~20]{eisenbud}. 

\begin{itemize}
\item If there exist ideals $A_1, \dots, A_t$ of $R$ such that 
\begin{equation*}
M \simeq R/A_1 \oplus \cdots \oplus R/A_t,
\end{equation*}
then we have
\begin{equation*}
\Fitt_R(M) = A_1 \cdots A_t.
\end{equation*}
\item Let $T$ be an $R$-algebra. We have
\begin{equation*}
\Fitt_T(M \otimes_R T) = \Fitt_R(M) T.
\end{equation*}
\item Let $N$ be another finitely generated $R$-module. We have
\begin{equation*}
\Fitt_R(M \oplus N) = \Fitt_R(M) \Fitt_R(N).
\end{equation*}
\end{itemize}

\begin{lemma}\label{lem:p1}
Let $M$ be a finite $\mathcal{O}_G^-$-module. Then
\begin{equation*}
|M| = |(\mathcal{O}_G^-/\Fitt_{\mathcal{O}_G^-}(M))|. 
\end{equation*}
\end{lemma}

\begin{proof}
We have
\begin{equation*}
(\mathcal{O}_G^- : \Fitt_{\mathcal{O}_G^-}(M)) = \prod_{\xi \in X_{\rm odd}} (e_\xi \mathbb{Z}[G] : e_\xi \Fitt_{\mathcal{O}_G^-}(M)) = \prod_{\xi \in X_{\rm odd}} (\mathbb{Z}[\xi] : \Fitt_{\mathbb{Z}[\xi]}(e_\xi M)). 
\end{equation*}
Fix $\xi \in X_{\rm odd}$. Since $e_\xi M$ is a finite $\mathbb{Z}[\xi]$-module, there exist ideals $\mathfrak{a}_1, \dots, \mathfrak{a}_r$ such that 
\begin{equation*}
e_\xi M = \mathbb{Z}[\xi]/\mathfrak{a}_1 \oplus \cdots \oplus \mathbb{Z}[\xi]/\mathfrak{a}_r. 
\end{equation*}
Therefore $\Fitt_{\mathbb{Z}[\xi]}(e_\xi M) = \mathfrak{a}_1 \cdots \mathfrak{a}_r$ and 
\begin{equation*}
(\mathbb{Z}[\xi] : \Fitt_{\mathbb{Z}[\xi]}(e_\xi M)) = N_{\mathbb{Q}(\xi)/\mathbb{Q}}(\mathfrak{a}_1 \cdots \mathfrak{a}_r) = |e_\xi M|.
\end{equation*}
It follows that $(\mathcal{O}_G^- : \Fitt_{\mathcal{O}_G^-}(M)) = \prod_{\xi \in X_{\rm odd}} |e_\xi M| = |M|$.
\end{proof}

\begin{lemma}\label{lem:p2}
Let $M$ be a finite $\mathbb{Z}[G]^-$-module. Let $p$ be a prime number not dividing $|G|$ and let $\psi$ be an odd irreducible $\mathbb{Z}_p$-character. Then
\begin{equation*}
|M^\psi| = |(\mathbb{Z}[G]^-/\Fitt_{\mathbb{Z}[G]^-}(M))^\psi| =  |(\mathcal{O}_G^-/\Fitt_{\mathcal{O}_G^-}(M))^\psi|. 
\end{equation*}
\end{lemma}

\begin{proof} 
We have $(\Fitt_{\mathbb{Z}[G]^-}(M))^\psi = \Fitt_{\mathbb{Z}_p[\psi]}(M^\psi)$. Since $M^\psi$ is a finite $\mathbb{Z}_p[\psi]$-module, there exist integers $c_1, \dots, c_r \geq 1$ such that
\begin{equation*}
M^\psi \simeq \bigoplus_{i=1}^r \mathbb{Z}_p[\psi]/\mathfrak{p}^{c_i} 
\end{equation*}
where $\mathfrak{p}$ is the prime ideal of $\mathbb{Z}_p[\psi]$. Then $\Fitt_{\mathbb{Z}[G]^-}(M)^\psi = \mathfrak{p}^{c}$ with $ c := c_1 + \cdots + c_r$ and therefore $|(\mathbb{Z}[G]^-/\Fitt_{\mathbb{Z}[G]^-}(M))^\psi| = (\mathbb{Z}_p[\xi] :\mathfrak{p}^c) = |M^\psi|$. The last equality is clear since $(\mathcal{O}_G^-)^\psi = \mathbb{Z}_p[\psi]$. 
\end{proof}

In what follows we will also use repeatedly the Tate cohomology of finite cyclic groups, see \cite[\S IX.1]{lang:book}. Let $A$ be a finite cyclic group with generator $a$ and let $M$  be a $\mathbb{Z}[A]$-module. The zero-th and first group of cohomology are defined by
\begin{equation*}
\hat{H}^0(A, M) := M^A/N_A(M) \quad\text{and}\quad \hat{H}^1(A, M) := \Ker (N_A: M \to M)/(1-a) M
\end{equation*}
where $N_A := \sum_{b \in A} b$ and $M^A$ is the submodule of elements in $M$ fixed by $A$. Let $N$ and $P$ be two other $\mathbb{Z}[A]$-modules such that the following short sequence is exact:
\begin{equation*}\xymatrix{
1 \ar[r] & M \ar[r] & N \ar[r] & P \ar[r] & 1.
}\end{equation*}
Then the hexagon below is also exact.
\begin{equation}\label{eq:hexa}\xymatrix{
 & \hat{H}^0(A, M) \ar[rd] &  \\
\hat{H}^1(A, P) \ar[ru] &  & \hat{H}^0(A, N) \ar[d] \\
\hat{H}^1(A, N) \ar[u] & & \hat{H}^0(A, P) \ar[ld] \\ 
& \hat{H}^1(A, M) \ar[lu] & \\
}\end{equation}
The Herbrand quotient of $M$ is defined by 
\begin{equation*}
Q(A, M) := \frac{|\hat{H}^0(A, M)|}{|\hat{H}^1(A, M)|}. 
\end{equation*}
The Herbrand quotient is multiplicative, that is for an exact short sequence as above, we have $Q(A, N) = Q(A, M) \, Q(A, P)$. The following result plays a crucial r\^ole in the next sections. It is a direct consequence of \cite[Cor.~IX.4.2]{lang:book}.
\begin{lemma}\label{lem:qu}
Let $E/F$ be a quadratic extension with Galois group $T$. Let $R \geq 0$ be the number of real places in $F$ that becomes complex in $E$. Then we have
\begin{equation*}
Q(T, U_E) = 2^{R-1}.  \tag*{\hspace{-1em}\qed}
\end{equation*}
\end{lemma}

We use this result in the following way. Assume that $R \geq 1$. Write $\bar{U}_F$ and $\bar{U}_E$ for the group of units of $F$ and $E$ respectively modulo $\{\pm 1\}$. Then we have
\begin{equation*}
\hat{H}^0(T, U_E) = \frac{U_F}{\mathcal{N}_{E/F}(U_E)} =  \{\pm 1\} \frac{\bar{U}_F}{\mathcal{N}_{E/F}(\bar{U}_E)}
\end{equation*}
since $-1$ cannot be a norm in $E/F$. It follows from the lemma that $|\hat{H}^0(T, U_E)|$ is divisible by $2^{R-1}$ and therefore
\begin{equation}\label{eq:bounde}
2^{R-2} \mid (\bar{U}_F:\mathcal{N}_{E/F}(\bar{U}_F)).  
\end{equation}

In some cases we will not be able to get non trivial lower bounds with that method, but still be able to deduce that $\hat{H}^1(T, U_E)$ is trivial. In this situation, we have the following lemma.
\begin{lemma}\label{lem:h11}
Let $E/F$ be a quadratic extension with Galois group $T$. Assume that $\hat{H}^1(T, U_E)$ is trivial. Then either $E/F$ is unramified at finite places or there exists an element of order $2$ in the kernel of the norm map from $\Cl_E$ to $\Cl_F$. 
\end{lemma}

\begin{proof}
Consider the submodules of elements fixed by $T$ in the short exact sequence 
\begin{equation*}\xymatrix{
1 \ar[r] & U_E \ar[r] & E^\times \ar[r] & P_E \ar[r] & 1. 
}\end{equation*}
We get
\begin{equation*}\xymatrix{
1 \ar[r] & U_F \ar[r] & F^\times \ar[r] & P_E^T \ar[r] & \hat{H}^1(T, U_E) \ar[r] & \cdots. 
}\end{equation*}
Since $\hat{H}^1(T, U_E) = 1$ by hypothesis, it follows that $P_F \simeq P_E^T$. The isomorphism is the natural map that sends $\mathfrak{a} \in P_F$ to $\mathfrak{a}\mathbb{Z}_E \in P_E^T$. Assume that there is a prime ideal $\mathfrak{p}$ of $F$ that ramifies in $E/F$. Let $\mathfrak{P}$ be the unique prime ideal of $E$ above $\mathfrak{p}$ and let $h \geq 1$ be the order of $\mathfrak{P}$ in $\Cl_E$. Since $\mathfrak{P}^h \in P_E^T$, there exists a principal ideal $\mathfrak{a} \in P_F$ such that $\mathfrak{P}^h = \mathfrak{a}\mathbb{Z}_E$. Clearly $\mathfrak{a}$ is a power of $\mathfrak{p}$. Looking at valuations at $\mathfrak{P}$, it follows that $h$ is even. We set $\mathfrak{C} := \mathfrak{P}^{h/2}$. Its class is an element of order $2$ in $\Cl_E$. But $\mathcal{N}_{E/F}(\mathfrak{C}) = \mathfrak{p}^{h/2} = \mathfrak{a}$ is a principal ideal. This concludes the proof.  
\end{proof}

To conclude this section we prove the conjecture in our settings when $K/k$ is a quadratic extension. This result is proved in full generality in \cite[Th.~IV.5.4]{tate:book}.

\begin{theorem}\label{th:quadratic}
Let $K/k$ be a quadratic extension and $S \supset S(K/k)$ be a finite set of places of $k$ satisfying \textup{(A1)}, \textup{(A2)} and \textup{(A3)}. Then the abelian rank one Stark conjecture is satisfied for the extension $K/k$ and the set $S$ with the Stark unit being the unique solution, up to trivial units, of \textup{(P1)} and \textup{(P2)}. Moreover the Stark unit is a square in $K$ if  and only if $e+t_S+c \geq 1$ where $c$ is the $2$-valuation of the order of $\Cl_K^-$. In particular, if $d \geq 4$ then it always a square and, in fact, it is a $2^{d-3}$-th power. It is also a square if $d = 3$ and the extension $K/k$ is ramified at some finite prime.
\end{theorem}

\begin{proof}
The only non trivial element of $G$ is $\tau$. Let $\chi$ be the character that sends $\tau$ to $-1$. It is the only non trivial character of $G$ and also the only odd character. We have $\mathbb{Z}[G]^- = \mathcal{O}_G^- = e^- \mathbb{Z} \simeq \mathbb{Z}$. In particular, using Proposition~\ref{prop:iso}, it is direct to see that there exists $\bar\theta \in U_K^-$ such that $U_K^- = \mathbb{Z} \cdot \bar\theta$. Define
\begin{equation*}
\bar\eta := 2^{e+t_S} |\Cl_K^-|  \cdot \bar\theta.
\end{equation*}
From its construction, it is clear that $\bar\eta$ satisfies (P1) and (P2). It follows from Proposition~\ref{prop:prodform}, and replacing $\eta$ by $\eta^{-1}$ if necessary, that 
\begin{equation*}
\frac{1}{2} \sum_{g \in G} \chi(g) \log |\eta^g| =  L'_{K/k}(0, \chi). 
\end{equation*}
This proves part ~(1)  of the conjecture. Part~(3) is direct by construction. It remains to prove part~(2). But $(\tau - 1) \cdot \bar\eta = -2 \cdot \bar\eta$ so part (2) follows and the conjecture is proved in this case. Finally, from its definition, it is clear that $\eta$ is a $2^r$-th power in $K^\times$ if and only if $e+t_S+c \geq r$. Now, by~\eqref{eq:bounde}, we have $e \geq d - 3$ and therefore the Stark unit is always a square if $d \geq 4$. Assume that $d = 3$ and that $\eta$ is not a square. Then $e = 0$ and $|\hat{H}^0(G, U_K)| = 2$. From Lemma~\ref{lem:qu}, we get $\hat{H}^1(G, U_K) = 1$ and therefore, since $c = 0$, the extension $K/k$ is unramified at finite places by Lemma~\ref{lem:h11} 
\end{proof}

When $d = 2$, there exist extensions for which the Stark unit is a square and extensions for which it is not a square. Using the PARI/GP system \cite{PARI}, we find the following examples.\footnote{PARI/GP was also used to find the examples given in the next two sections.} Let $k := \mathbb{Q}(\sqrt{5})$ and let $v_1, v_2$ denote the two infinite places of $k$ with $v_1(\sqrt{5}) < 0$ and $v_2(\sqrt{5}) > 0$. Let $K$ be the ray class field modulo $\mathfrak{p}_{11} v_2$ where $\mathfrak{p}_{11} := (1/2 +  3\sqrt{5}/2)$ is one of the two prime ideals above $11$. Then $K/k$ is a quadratic extension that satisfies (A1), (A2) and (A3) with $S := S(K/k)$, and one can prove that the corresponding Stark unit is not a square. Now, on the other hand, let $K$ be the ray class field modulo $\sqrt{5}\mathfrak{q}_{11} v_1$, where $\mathfrak{q}_{11} := (1/2 - 3\sqrt{5}/2)$ is the other prime ideal above $11$. Then $K/k$ is a quadratic extension that satisfies (A1), (A2) and (A3) with $S := S(K/k)$ and, in this case, the Stark unit is a square. When $d = 3$ and $K/k$ is unramified both cases are possible. Indeed, let $k := \mathbb{Q}(\alpha)$ where $\alpha^3 - \alpha^2 - 13\alpha + 1 = 0$. It is a totally real cubic field. Let $v_1, v_2, v_3$ be the three infinite places of $k$ with $v_1(\alpha) \approx -3.1829$, $v_2(\alpha) \approx 0.0765$ and $v_3(\alpha) \approx 4.1064$. Let $K$ be the ray class field of $k$ of conductor $\mathbb{Z}_k v_2v_3$. Then $K/k$ is a quadratic extension that satisfies (A1), (A2) and (A3) with $S := S(K/k)$, and that is unramified at finite places. One can prove in this setting that the Stark unit is not a square. On the other hand, let $k := \mathbb{Q}(\beta)$ with $\beta^3 - \beta^2 - 24\beta - 35 = 0$. It is a totally real cubic field. Let $v_1, v_2, v_3$ be the three infinite places of $k$ with $v_1(\alpha) \approx -3.0999$, $v_2(\alpha) \approx -1.8861$, and $v_3(\alpha) \approx  5.9860$. Let $K$ be the unique quadratic extension $k$ of conductor $\mathbb{Z}_k v_2v_3$. Then $K/k$ satisfies (A1), (A2) and (A3) with $S := S(K/k)$ and is unramified at finite places. One can prove that $k$ is principal and the class number of $K$ is $2$. Therefore the Stark unit in this case is a square. 

\section{Cyclic quartic extensions}

The goal of this section is to prove the following result.

\begin{theorem}\label{th:quartic}
Let $K/k$ be a cyclic quartic extension and $S \supset S(K/k)$ be a finite set of places of $k$ satisfying \textup{(A1)}, \textup{(A2)} and \textup{(A3)}. Then there exists $\bar\eta \in U_K^-$ satisfying \textup{(P1)} and \textup{(P2)}. Furthermore, $\bar\eta$ is unique up to a trivial unit, satisfies for all $\chi \in \hat{G}$
\begin{equation*}
\left|L'_{K/k,S}(0, \chi)\right| = \frac{1}{2} \left|\sum_{g \in G} \chi(g) \log |\eta^g|\right|
\end{equation*}
and the extension $K(\sqrt{\eta})/k$ is abelian.
\end{theorem}

\begin{proof}
Denote by $\gamma$ a generator of $G$, therefore $\tau = \gamma^2$. Let $\chi$ be the character of $G$ such that $\chi(\gamma) = i$ and let $\xi := \chi + \chi^3$ be the only element in $X_\mathrm{odd}$. From the results of Section~4, we have
\begin{equation*}
\mathbb{Q}[G]^- = e^- \mathbb{Q}[G] \simeq \mathbb{Q}(i)
\end{equation*}
where the isomorphism sends any element of $x \in \mathbb{Q}[G]^-$, written uniquely as $x = e^- (a + b\gamma)$ for $a, b \in \mathbb{Q}$, to $\chi(x) = a+bi$. In particular, we have $\mathbb{Z}[G]^- = \mathcal{O}_G^- \simeq \mathbb{Z}[i]$ and $\mathbb{Z}[G]^-$ is a principal ring. By Proposition~\ref{prop:iso}, this implies that there exists $\bar\theta \in U_K^-$ such that $U_K^- = \mathbb{Z}[G]^- \cdot \bar\theta$. 

We now prove the unicity of the solution. Assume that $\bar\eta$ and $\bar\eta'$ are two solutions to (P1) and (P2). By Corollary~\ref{cor:unicity}, there exists a unique $2$-unit $u$ in $\mathbb{Q}[G]^-$ such that $\bar\eta' = u \cdot \bar\eta$. Let $\mathfrak{p}_2 := (i+1)\mathbb{Z}[i]$ be the unique prime ideal above $2$ in $\mathbb{Z}[i]$. Let $n := v_{\mathfrak{p}_2}(\chi(u))$. Assume, without loss of generality, that $n \geq 0$ (otherwise, exchange $\bar\eta$ and $\bar\eta'$ and replace $u$ by $u^{-1}$) and therefore $\bar\eta' \in \mathbb{Z}[G]^- \cdot \bar\eta$. Let $x \in \mathbb{Z}[G]^-$ be such that $\bar\eta = x \cdot \bar\theta$. We have
\begin{align*}
(\mathbb{Z}[G]^- \cdot \bar\eta : \mathbb{Z}[G]^- \cdot \bar\eta') & = (x \mathbb{Z}[G]^- : u x \mathbb{Z}[G]^-) \\
& = (\chi(x) \mathbb{Z}[i] : \chi(u) \chi(x) \mathbb{Z}[i]) \\
& = |\chi(u)| = 2^n.
\end{align*}
Therefore $n = 0$ and $u$ is a unit. Since $\mathbb{Z}[G]^-$ has only trivial units, it follows that $u$ is a trivial unit. This proves the unicity statement. 

Next we prove that there exist solutions to (P1) and (P2). Let $\mathcal{F} := \Fitt_{\mathbb{Z}[G]^-}(\Cl_K^-)$ be the Fitting ideal of $\mathrm{Cl}_K^-$ as a $\mathbb{Z}[G]^-$-module. Let $f$ be a generator of $\mathcal{F}$. We set $\bar\eta := f \, (\gamma+1)^{e+t_S} \cdot \bar\theta$. We have by Lemma~\ref{lem:p1}
\begin{equation*}
(U_K^- : \mathbb{Z}[G] \cdot \bar\eta) = 2^{e+t_S} (\mathbb{Z}[G]^- : \mathcal{F}) = 2^{e+t_S} |\Cl_K^-|.
\end{equation*}
Thus $\bar\eta$ is a solution to (P1). In the same way it follows directly from Lemma~\ref{lem:p2} that it is a solution to (P2). 

Now, since $\bar\eta \in U_K^-$, we have for $\nu = \chi_0$, the trivial character, or $\nu = \chi^2$ that
\begin{equation*}
\frac{1}{2} \sum_{g \in G} \nu(g) \log |\eta^g| = 0.
\end{equation*}
On the other hand, $L'_{K/k, S}(\nu, 0) = 0$ follows directly from \cite[Prop.~I.3.4]{tate:book}. From Proposition~\ref{prop:prodform}, using the fact that $\chi^3 = \bar\chi$, we get\footnote{Note that is easy to see that both sides of \eqref{eq:factform} are positive in this case.}
\begin{multline*}
\left|L'_{K/k,S}(0, \chi)\right|^2 = L'_{K/k,S}(0, \chi) L'_{K/k,S}(0, \chi^3) \\
 = \left(\frac{1}{2} \sum_{g \in G} \chi(g) \log |\eta^g|\right) \left(\frac{1}{2} \sum_{g \in G} \chi^3 (g) \log |\eta^g|\right) \\
= \left|\frac{1}{2} \sum_{g \in G} \chi(g) \log |\eta^g|\right|^2
\end{multline*}
and the equality to be proved follows by taking square-roots. 

Finally, to prove that $K(\sqrt{\eta})/k$ is abelian, we need to prove that $(\gamma - 1) \cdot \bar\eta \in 2 \cdot U_K^-$ by \cite[Prop.~IV.1.2]{tate:book}. This is equivalent to prove that
\begin{equation*}
(i-1)(i+1)^{e+t_S} \chi(f) \subset 2 \mathbb{Z}[i],
\end{equation*}
that is one of the following assertions is satisfied: $e \geq 1$, $t_S \geq 1$ or $2$ divides $|\Cl_K^-|$. We have $e \geq 2d-4$ by~\eqref{eq:bounde} and therefore the result is proved if $d \geq 3$. Assume that $d = 2$ and $e = 0$. Then it follows by Lemma~\ref{lem:qu} that $\hat{H}^1(T, U_K) = 1$ where $T := \langle \tau\rangle$. By Lemma~\ref{lem:h11} this implies that either $2$ divides $|\Cl_K^-|$ and the result is proved, or $K/K^+$ is unramified at finite places. Assume the latter. At least one prime ideal of $k$ ramifies in $K$ by the proof of \cite[Lem.~2.8]{stark12} since $k$ is a quadratic field. By (A3) this prime ideal is inert in $K/K^+$, thus $t_S \geq 1$. This concludes the proof. 
\end{proof}

A consequence of this result is that we can say quite precisely when the Stark unit, it it exists, is a square in that case. The result is very similar to the situation in the quadratic case (see Theorem~\ref{th:quadratic}). 

\begin{corollary}\label{cor:square4}
Under the hypothesis of the theorem and assuming that the Stark unit exists, then it is a square in $K$ if and only if $e+t_S+c \geq 2$ where $c$ is the $2$-valuation of $|\Cl_K^-|$. In particular, if $d \geq 3$ then it is always a square and, in fact, it is a $2^{d-2}$-th power. 
\end{corollary}

\begin{proof}
We prove the equivalence. The inequality is satisfied when the Stark unit $\varepsilon$ is a square by Corollary~\ref{cor:whensquare}. Now assume that the inequality is satisfied. By the unicity statement of the theorem, we have $\bar\varepsilon = \bar\eta$ (replacing $\eta$ by one of its conjugate if necessary). From the proof of the theorem, we see that $\bar\eta$ belongs to $2^r \cdot U_K^-$ if and only if $(i+1)^{e+t_S} \chi(f) \in 2^r \mathbb{Z}[i]$. Taking valuation at $\mathfrak{p}_2$, the only prime ideal above $2$, we see that it is equivalent to $e+t_S+c \geq 2r$. This proves the first assertion. Now, to prove the second assertion, we see that $e \geq 2d-4$ by \eqref{eq:bounde}. Therefore $\bar{\eta}$ lies in $2^{d-2} \cdot U_K^-$. This proves the result. 
\end{proof}
 
When $d = 2$ it is possible to find examples for which the Stark unit, if it exists, is a square and examples for which it is not a square. For example, let $k := \mathbb{Q}(\sqrt{5})$ and let $v_1, v_2$ denote the two infinite places of $k$ with $v_1(\sqrt{5}) < 0$ and $v_2(\sqrt{5}) > 0$. Let $K$ be the ray class field modulo $\mathfrak{p}_{29} v_1$ where $\mathfrak{p}_{29} := (11/2 - \sqrt{5}/2)$ is one of the two prime ideals above $29$. Then $K/k$ is a cyclic quartic extension that satisfies (A1), (A2) and (A3) with $S := S(K/k)$ and one can prove that, if it exists, the Stark unit is not a square. Now, on the other hand, let $K$ be the ray class field modulo $\sqrt{5}\mathfrak{p}_{41} v_1$ where $\mathfrak{p}_{41} := (13/2 - \sqrt{5}/2)$ is one of the two prime ideals above $41$. Then $K/k$ is a cyclic quartic extension that satisfies (A1), (A2) and (A3) with $S := S(K/k)$, but one can prove that, in this case, the Stark unit, if it exists, is a square.
 
\section{Cyclic sextic extensions}

In this final section we study the case when $K/k$ is a cyclic sextic extension. We will need some additional assumptions to be able to prove that there exists solutions to (P1) and (P2). 

\begin{theorem}\label{th:sextic}
Let $K/k$ be a cyclic sextic extension such that \textup{(A1)}, \textup{(A2)} and \textup{(A3)} are satisfied with $S := S(K/k)$  . Assume also that $3$ does not divide the order of $\Cl_K$ and that no prime ideal above $3$ is wildly ramified in $K/k$. Let $F$ be the quadratic extension of $k$ contained in $K$. Then there exists $\bar\eta \in U_K^-$ satisfying \textup{(P1)} and \textup{(P2)} and such that $\mathcal{N}_{K/F}(\eta)$ is the Stark unit for the extension $F/k$ and the set of places $S$. Furthermore, $\bar\eta$ is unique up to multiplication by an element of $\Gal(K/F)$, satisfies for all $\chi \in \hat{G}$
\begin{equation*}
\left|L'_{K/k,S}(0, \chi)\right| = \frac{1}{2} \left|\sum_{g \in G} \chi(g) \log |\eta^g|\right|,
\end{equation*}
and the extension $K(\sqrt{\eta})/k$ is abelian.
\end{theorem}

\begin{proof}
 Let $\gamma$ be a generator of the Galois group $G$, thus $\tau = \gamma^3$. Let $\chi$ be the character that sends $\gamma$ to $-\omega$ where $\omega$ is a fixed primitive third root of unity. It is a generator of the group of characters of $G$. We have $X_{\rm odd} = \{\xi_2, \xi_6\}$ where $\xi_2 := \chi^3$ and $\xi_6 := \chi + \chi^5$. The corresponding idempotents are
\begin{equation*}
e_{\xi_2} = \frac{1}{6}(1-\gamma^3)(1+\gamma^2+\gamma^4)\quad  \text{and}\quad
e_{\xi_6} = \frac{1}{6}(1-\gamma^3)(2-\gamma^2-\gamma^4).
\end{equation*}
We have the isomorphism
\begin{equation}\label{eq:q6iso}
\mathbb{Q}[G]^- = e_{\xi_2} \mathbb{Q}[G] + e_{\xi_6} \mathbb{Q}[G] \cong \mathbb{Q} \oplus \mathbb{Q}(\omega).
\end{equation}
Let $\sigma := \gamma^2$ and let $H$ be the subgroup of order $3$ generated by $\sigma$. Then we have $G = \langle \tau \rangle \times H$ and the projection map on $H$ extends to an isomorphism between $\mathbb{Q}[G]^-$ and $\mathbb{Q}[H]$, that restricts to an isomorphism between $\mathbb{Z}[G]^-$ and $\mathbb{Z}[H]$. From now on, we will identify $\mathbb{Q}[G]^-$ and $\mathbb{Q}[H]$. Note that, with that identification, both act in the same way on $U_K^-$, $U_K^- \otimes \mathbb{Q}$, $\Cl_K^-$, etc. Let $e_0$ and $e_1$ be the image by the projection map of $e_{\xi_2}$ and $e_{\xi_6}$. Then $e_0 = \frac{1}{3}(1+\sigma+\sigma^2)$ is the idempotent of the trivial character of $H$ and $e_1 = \frac{1}{3}(2-\sigma-\sigma^2)$ is the sum of the idempotents of the two non trivial characters of $H$. The main difference between this case and the quartic case is the fact that the isomorphism between $\mathbb{Q}[H]$ and $\mathbb{Q} \oplus \mathbb{Q}(\omega)$ does not restrict to an isomorphism between $\mathbb{Z}[G]^-$ and $\mathbb{Z} \oplus \mathbb{Z}[\omega]$. In particular, $\mathbb{Z}[G]^-$ is not a principal ring. Because of that the proof is somewhat more intricate than in the quartic case. We will therefore proceed by proving a series of different claims. First, we define 
\begin{equation}\label{eq:oiso}
\mathcal{O} := e_0 \mathbb{Z}[H] + e_1 \mathbb{Z}[H] \simeq \mathbb{Z} \oplus \mathbb{Z}[\omega]. 
\end{equation}
Note that, by the above identification, we have $\mathcal{O}_G^- \cong \mathcal{O}$. 

\begin{claim}\label{claim:oprincipal}
The ring $\mathcal{O}$ is principal and contains $\mathbb{Z}[H]$ with index 3. 
\end{claim}

Let $\mathcal{I}$ be an ideal of $\mathcal{O}$. Then $e_0\mathcal{I}$ is an ideal of $e_0\mathbb{Z} \simeq \mathbb{Z}$. Thus there exists $a \in \mathbb{Z}$ such that  $e_0\mathcal{I} = a e_0\mathbb{Z}[H]$. In the same way, $e_1\mathcal{I}$ is an ideal of $e_1\mathbb{Z} \simeq \mathbb{Z}[\omega]$. Since $\mathbb{Z}[\omega]$ is a principal ring, there exists $b, c \in \mathbb{Z}$ such that  $e_1\mathcal{I} = e_1(b+c\sigma)\mathbb{Z}[H]$. One verify readily that $e_0 a + e_1 (b + c\sigma)$ is a generator of $\mathcal{I}$. To conclude the proof of Claim~\ref{claim:oprincipal}, we note that $\mathcal{O}/\mathbb{Z}[H] = \langle e_0+\mathbb{Z}[H]\rangle = \langle e_1+\mathbb{Z}[H]\rangle$ clearly has order $3$. 

\begin{claim}\label{claim:mainid} 
Let $\mathcal{A}$ be an ideal of $\mathbb{Z}[H]$ of finite index. Then there exists $g \in \mathcal{A}$ such that
\begin{equation}\label{eq:index}
\mathcal{O}/\mathcal{AO} \simeq \mathbb{Z}[H]/g\mathbb{Z}[H].
\end{equation}
Furthermore, $\mathcal{A} = g\mathbb{Z}[H]$ if $\mathcal{A}$ is a principal ideal. Otherwise $(\mathcal{A}:g\mathbb{Z}[H]) = 3$. 
\end{claim}

We prove this claim by considering the two cases: $\mathcal{AO} \not= \mathcal{A}$ and $\mathcal{AO} = \mathcal{A}$.

\begin{subclaim}\label{claim:aprincipal}
Assume that $\mathcal{AO} \not= \mathcal{A}$. Then $\mathcal{A}$ is a principal ideal.
\end{subclaim}

Let $g' = e_0a + e_1(b+c\sigma)$ be a generator of the principal ideal $\mathcal{AO}$ of $\mathcal{O}$. If $e_1(b+c\sigma) \in \mathcal{A}$, then $e_1\mathcal{A}  = e_1(b+c\sigma)\mathbb{Z}[H] \subset \mathcal{A}$ and it follows that $\mathcal{A} = \mathcal{AO}$, a contradiction. Therefore $\mathcal{AO}/\mathcal{A} = \langle e_1(b+c\sigma) + \mathcal{A} \rangle$ has order $3$. Thus one of the three elements: $e_0a + e_1(b+c\sigma)$, $e_0a - e_1(b+c\sigma)$ or $e_0a$  belongs to $\mathcal{A}$. It cannot be $e_0a$ since that would imply, as above, that $\mathcal{A} = \mathcal{AO}$. Denote by $g$ the one element between $e_0a \pm e_1(b+c\sigma)$ that lies in $\mathcal{A}$. Clearly we still have $g\mathcal{O} = \mathcal{AO}$. Now, $g$ is not a zero divisor since $\mathcal{A}$ has finite index in $\mathbb{Z}[H]$, so we have $(g\mathcal{O}:g\mathbb{Z}[H]) = (\mathcal{O}:\mathbb{Z}[H]) = 3$. Therefore we get
\begin{equation*}
(\mathcal{A}:g\mathbb{Z}[H]) = \frac{(g\mathcal{O}:g\mathbb{Z}[H])}{(\mathcal{AO}:\mathcal{A})} = 1
\end{equation*}
and $\mathcal{A}=g\mathbb{Z}[H]$. Equation \eqref{eq:index} follows in that case from the equality
\begin{equation}\label{eq:indexAO}
(\mathcal{O}:\mathcal{AO}) (\mathcal{AO}:\mathcal{A}) = (\mathcal{O}: \mathbb{Z}[H]) (\mathbb{Z}[H]: \mathcal{A})
\end{equation}
and the fact that $(\mathcal{AO}:\mathcal{A}) = (\mathcal{O}: \mathbb{Z}[H])$ by the above. 

\begin{subclaim}\label{claim:anonprincipal}
Assume that $\mathcal{AO} = \mathcal{A}$. Then $\mathcal{A}$ is not a principal ideal, but there exists $g \in \mathcal{A}$ such that $(\mathcal{A}: g\mathbb{Z}[H]) = 3$.
\end{subclaim}

Let $g$ be a generator of the principal ideal $\mathcal{AO}$ of $\mathcal{O}$. Since $\mathcal{AO} =\mathcal{A}$, $g$ lies in $\mathcal{A}$ and we compute as above
\begin{equation*}
(\mathcal{A}:g\mathbb{Z}[H]) =  (\mathcal{AO}:g\mathcal{O})(g\mathcal{O}:g\mathbb{Z}[H]) = 3. 
\end{equation*}

Since $(\mathcal{O}: \mathbb{Z}[H])(\mathbb{Z}[H]:\mathcal{A}) = 3(\mathbb{Z}[H]:\mathcal{A}) = (\mathcal{A}:g\mathbb{Z}[H])(\mathbb{Z}[H]:\mathcal{A}) = (\mathbb{Z}[H]:g\mathbb{Z}[H])$ and $(\mathcal{AO}:\mathcal{A}) = 1$, Equation~\eqref{eq:index} follows from \eqref{eq:indexAO}. It remains to prove that $\mathcal{A}$ cannot be principal in that case. In order to prove this, we need another result. Let $x \in \mathcal{O}$. By the isomorphism in \eqref{eq:oiso}, it corresponds to a pair $(x_0, x_1)$ in $\mathbb{Z} \oplus \mathbb{Z}[\omega]$. We define the norm of $x$ as the following quantity
\begin{equation*}
\text{Norm}(x) := |x_0| \, N_{\mathbb{Q}(\omega)/\mathbb{Q}}(x_1).
\end{equation*}
Note that we recover the usual definition of the norm of $\mathbb{Q}[H]$ as a $\mathbb{Q}$-algebra. The proof of the following claim is straightforward and is left to the reader.

\begin{claim}\label{claim:norm}
Let $x \in \mathcal{O}$ with $\text{Norm}(x) \not= 0$. Then $(\mathcal{O}:x\mathcal{O}) = \text{Norm}(x)$. If furthermore $x \in \mathbb{Z}[H]$ then $(\mathbb{Z}[H]:x\mathbb{Z}[H]) = \text{Norm}(x)$.
\end{claim}

\noindent We now finish the proof of Claim~\ref{claim:anonprincipal}. Assume that $\mathcal{A}$ is principal, say $\mathcal{A} = h\mathbb{Z}[H]$. Then there exists $z \in \mathbb{Z}[H]$ such that $g = hz$ and we have $(\mathcal{O}:z\mathcal{O}) = 3$. Thanks to the above claim, we can explicitly compute all the elements $z \in \mathcal{O}$ such that $(\mathcal{O}:z\mathcal{O}) = 3$. There are the elements  $z = e_0 a + e_1 (b + c\sigma)$ with $a = \pm 1$ and $b + c\sigma \in \{\pm(1+2\sigma), \pm(2+\sigma), \pm (1-\sigma)\}$, or $a = \pm 3$ and $b + c\sigma \in \{\pm 1, \pm \sigma, \pm (1+\sigma)\}$. One can compute all possibilities and check that none of those belong to $\mathbb{Z}[H]$. This gives a contradiction and concludes the proof of Claim~\ref{claim:anonprincipal} and of Claim~\ref{claim:mainid}. 

\smallskip

We now turn to the $\mathbb{Z}[H]$-structure of $U_K^-$. The principal result is the following claim that we will prove in several steps.

\begin{claim}\label{claim:ukprincipal}
There exists $\bar\theta \in U_K^-$ such that $U_K^- = \mathbb{Z}[H] \cdot \bar\theta$.  
\end{claim}

Let $\bar\theta' \in U_K^-$ be such that $U_K^- \otimes \mathbb{Q} = \mathbb{Q}[H] \cdot \bar\theta'$. Note that the existence of $\bar\theta'$ follows from Proposition~\ref{prop:iso}. We define 
\begin{equation*}
\Lambda := \left\{ x \in \mathbb{Q}[H] : x \cdot \bar\theta' \in U_K^- \right\}.  
\end{equation*}
It is a fractional ideal of $\mathbb{Z}[H]$. The above claim is satisfied if and only if it is a principal ideal. Assume that this is not the case. Then, by the above, we have\footnote{Strictly speaking, the claims above are only for integral ideals of $\mathbb{Z}[H]$ but they admit obvious and direct generalizations to fractional ideals.}  $\Lambda\mathcal{O} = \Lambda$. Recall that $F$ denotes the subfield of $K$ fixed by $H$. It is a quadratic extension of $k$ and $\Gal(F/k) = \langle \tau\rangle$. We define $U_F^-$ as the kernel of norm map from $U_F/\langle \pm 1\rangle$ to $U_k/\langle \pm 1\rangle$. We have also $U_F^- = U_K^- \cap (F^\times/\langle \pm 1\rangle)$. Let $N_H := 1 + \sigma + \sigma^2$. It is the group ring element corresponding to the norm of the extensions $K/F$ and $K^+/k$. 

\begin{subclaim}\label{claim:nuk=3uf}
$\Lambda\mathcal{O} = \Lambda$ if and only if $N_H \cdot U_K^-  = 3 \cdot U_F^-$. If $\Lambda\mathcal{O} \not= \Lambda$, then $N_H \cdot U_K^- = U_F^-$.
\end{subclaim}

\noindent We have $\Lambda\mathcal{O} = \Lambda$ if and only if $e_0\Lambda \subset \Lambda$, that is $N_H \cdot U_K^- \subset 3 \cdot U_K^-$. Assume that it is the case. Let $\bar\delta \in U_K^-$ and set $\kappa := N_{K/F}(\delta) \in U_F$. Then the polynomial $X^3 - \kappa$ has a root, say $\nu$, in $U_K$. If $\nu$ does not belong to $F$ then all the roots of $X^3 - \kappa$ belongs to $K$ since $K/F$ is a Galois extension. It follows that $K$ contains the third roots of unity, a contradiction. Therefore $\bar\nu \in U_F^-$ and $N_H \cdot U_K^- \subset 3 \cdot U_F^-$. The other inclusion is trivial and the first assertion of the claim is proved. If $\Lambda\mathcal{O} \not= \Lambda$ then $3 \cdot U_F^- \varsubsetneq N_H \cdot U_K^- \subset U_F^-$. Since $U_F^-$ is a $\mathbb{Z}$-module of rank $1$, it follows that $N_H \cdot U_K^- = U_F^-$. The claim is proved. 

\smallskip

Let $\mathcal{S}$ be the set of prime ideals of $K$ that are totally split in $K/k$. Denote by $I_{K,\mathcal{S}}$ the subgroup of $I_K$, the group of ideals of $K$, generated by the prime ideals in $\mathcal{S}$. Then, by Chebotarev's theorem, the following short sequence is exact
\begin{equation*}\xymatrix{
1 \ar[r] & P_K \cap I_{K,\mathcal{S}}^{1-\tau} \ar[r] & I_{K,\mathcal{S}}^{1-\tau} \ar[r] & \Cl_K^{1-\tau} \ar[r] & 1
}\end{equation*}
where $P_K$ is the group of principal ideals of $K$. We take the Tate cohomology of this sequence for the action of $H$. Since $3$ does not divide the order of $\Cl_K$, it does not divide the order of $\Cl_K^{1-\tau}$ and $\hat{H}^0(\Cl_K^{1-\tau}) = \hat{H}^1(\Cl_K^{1-\tau}) = 1$. Note that here and in what follows, to simplify the presentation, we drop the group $H$ in the notation of the cohomology groups and write $\hat{H}^i(M)$ instead of $\hat{H}^i(H, M)$ for $M$ a $\mathbb{Z}[H]$-module. It follows from the exact hexagon \eqref{eq:hexa} for the above exact sequence that $\hat{H}^i(P_K \cap I_{K,\mathcal{S}}^{1-\tau}) \simeq \hat{H}^i(I_{K,\mathcal{S}}^{1-\tau})$ for $i = 0, 1$. Let $\mathfrak{A} \in P_K \cap I_{K,\mathcal{S}}^{1-\tau}$. There exist $\alpha \in K_\mathcal{S}^\times$, the subgroup of elements of $K^\times$ supported only by prime ideals in $\mathcal{S}$, and $\mathfrak{B} \in I_{K,\mathcal{S}}$ such that 
\begin{equation*}
\mathfrak{A} = (\alpha) = \mathfrak{B}^{1-\tau}. 
\end{equation*}
We apply $1-\tau$ to this equation
\begin{equation*}
\mathfrak{A}^{1-\tau} = (\alpha)^{1-\tau} = \mathfrak{B}^{(1-\tau)^2} = (\mathfrak{B}^{1-\tau})^2 = \mathfrak{A}^2. 
\end{equation*}
Therefore we have $(P_K \cap I_{K,\mathcal{S}}^{1-\tau})^2 \subset P_{K,\mathcal{S}}^{1-\tau}$ where $P_{K,\mathcal{S}}$ is the subgroup of principal ideals generated by the elements of $K^\times_\mathcal{S}$. It follows that the quotient $(P_K \cap I_{K,\mathcal{S}}^{1-\tau})/P_{K,\mathcal{S}}^{1-\tau}$ is killed by $2$ and therefore $\hat{H}^i(P_K \cap I_{K,\mathcal{S}}^{1-\tau}) = \hat{H}^i(P_{K,\mathcal{S}}^{1-\tau})$ for $i = 0$ or $1$. We have proved the following claim.

\begin{subclaim}\label{claim:isohi}
$\hat{H}^0(P_{K,\mathcal{S}}^{1-\tau}) \simeq \hat{H}^0(I_{K,\mathcal{S}}^{1-\tau})$ and $\hat{H}^1(P_{K,\mathcal{S}}^{1-\tau}) \simeq \hat{H}^1(I_{K,\mathcal{S}}^{1-\tau})$.
\end{subclaim}

\noindent Let $u \in U_K \cap (K_\mathcal{S}^\times)^{1-\tau}$. There exists $\alpha \in K_\mathcal{S}^\times$ such that $u = \alpha^{1-\tau}$. Therefore we get 
\begin{equation*}
u^{1-\tau} = \alpha^{(1-\tau)^2} = (\alpha^{1-\tau})^2 = u^2. 
\end{equation*}
Reasoning as above, this implies that $\hat{H}^i(U_K \cap (K_\mathcal{S}^\times)^{1-\tau}) = \hat{H}^i(U_K^{1-\tau}) = \hat{H}^i(U_K^-)$ for $i = 0, 1$. We now consider the short exact sequence
\begin{equation*}\xymatrix{
1 \ar[r] & U_K \cap (K_\mathcal{S}^\times)^{1-\tau} \ar[r] & (K_\mathcal{S}^\times)^{1-\tau} \ar[r] & P_{K,\mathcal{S}}^{1-\tau} \ar[r] & 1.
}\end{equation*}
Taking the Tate cohomology and using the above equalities, we extract the following exact sequence from the  exact hexagon~\eqref{eq:hexa} corresponding to this exact sequence
\begin{equation}\label{eq:exact}\xymatrix{
\cdots \ar[r] & \hat{H}^1(P_{K,\mathcal{S}}^{1-\tau}) \ar[r] & \hat{H}^0(U^-_K) \ar[r] & \hat{H}^0((K_\mathcal{S}^\times)^{1-\tau}) \ar[r] & \cdots \\
}\end{equation}

\noindent The next claim is just a reformulation of the first part of Claim~\ref{claim:nuk=3uf}.\\[-5pt]

\begin{subclaim}~\label{claim:h1uk=3}
$\Lambda\mathcal{O} = \Lambda$ if and only if $\hat{H}^0(U_K^-) \simeq \mathbb{Z}/3\mathbb{Z}$. 
\end{subclaim}

\noindent Assume the two followings claims for the moment. 

\begin{subclaim}\label{claim:H0PK=1}
$\hat{H}^1(P_{K,\mathcal{S}}^{1-\tau})$ is trivial.
\end{subclaim}

\begin{subclaim}\label{claim:H0K=1}
$\hat{H}^0((K_\mathcal{S}^\times)^{1-\tau})$ is trivial.
\end{subclaim}

\noindent By \eqref{eq:exact} we get that $\hat{H}^0(U_K^-) = 1$. Thus $\Lambda\mathcal{O} \not= \mathcal{O}$ by Claim~\ref{claim:h1uk=3} and therefore $\Lambda$ is principal by Claim~\ref{claim:aprincipal}, and Claim~\ref{claim:ukprincipal} follows. It remains to prove Claims~\ref{claim:H0PK=1} and \ref{claim:H0K=1}. We start with 
the proof of Claim~\ref{claim:H0PK=1}. By Claim~\ref{claim:isohi}, this is equivalent to prove that $\hat{H}^1(I_{K,\mathcal{S}}^{1-\tau})$ is trivial. We have as $\mathbb{Z}[H]$-modules
\begin{equation*}
I_{K,\mathcal{S}}^{1-\tau} = \prod_{\mathfrak{p}_0 \in \mathcal{S}_0}{\!\!\!}^{'} \Big(\prod_{\mathfrak{P} \mid \mathfrak{p}_0} \mathfrak{P}^{\mathbb{Z}}\Big)^{1-\tau} \simeq  \prod_{\mathfrak{p}_0 \in \mathcal{S}_0}{\!\!\!}^{'} (1-\tau) \mathbb{Z}[G] 
\end{equation*}
where $\mathcal{S}_0$ is the set of prime ideals of $k$ that splits completely in $K/k$, $\mathfrak{P}$ runs through the prime ideals of $K$ dividing $\mathfrak{p}_0$ and the $'$ indicates that it is a restricted product, that is the exponent of $\mathfrak{P}$ is zero for all but finitely many prime ideals. The  isomorphism comes from fixing a prime ideal above $\mathfrak{p}_0$ and the fact that $\mathfrak{p}_0$ is totally split in $K/k$. Therefore we have 
\begin{equation*}
\hat{H}^1(I_{K,\mathcal{S}}^{1-\tau}) = \prod_{\mathfrak{p}_0 \in \mathcal{S}_0}{\!\!\!}^{'} \hat{H}^1((1-\tau)\mathbb{Z}[G]) \simeq \prod_{\mathfrak{p}_0 \in \mathcal{S}_0}{\!\!\!}^{'} \hat{H}^1(\mathbb{Z}[H]). 
\end{equation*}
It is well-known that $\hat{H}^1(\mathbb{Z}[H]) = 1$, thus Claim \ref{claim:H0PK=1} is proved.

To prove Claim~\ref{claim:H0K=1}, we prove that the norm from $(K_\mathcal{S}^\times)^{1-\tau}$ to $(F_\mathcal{S}^\times)^{1-\tau}$ is surjective. Let $\alpha^{1-\tau} \in (F_\mathcal{S}^\times)^{1-\tau}$. By the Hasse Norm Principle, $\alpha^{1-\tau}$ is a norm in $K/F$ if and only if it is a norm in $K_\mathfrak{P}/F_\mathfrak{p}$ for all prime ideals $\mathfrak{P}$ of $K$ where $\mathfrak{p}$ denotes the prime ideal of $F$ below $\mathfrak{P}$. If $\mathfrak{p}$ splits in $K/F$, then $\alpha^{1-\tau}$ is trivially a norm in $K_\mathfrak{P}/F_\mathfrak{p}$. Assume now that $\mathfrak{p}$ is inert. It follows from the theory of local fields, see \cite[\S XI.4]{lang:book}, that the norm of $K_\mathfrak{P}/F_\mathfrak{p}$ is surjective on the group of units of $F_\mathfrak{p}$. But  $\alpha^{1-\tau}$ is a unit at $\mathfrak{P}$ since $\mathfrak{P} \not\in \mathcal{S}$,  and therefore it is a norm also in this case. Finally we assume that $\mathfrak{P}$ is ramified in $K/F$. Let $p$ be the rational prime below $\mathfrak{P}$. By hypothesis, $p \not= 3$ since $3$ is not wildly ramified in $K/k$. Write $\mu_\mathfrak{P}$, $\mathbb{U}_\mathfrak{P}$, $\mu_\mathfrak{p}$ and $\mathbb{U}_\mathfrak{p}$ for the group of roots of unity of order prime to $p$ and the group of principal units of $K_\mathfrak{P}$ and $F_\mathfrak{p}$ respectively. We have $\mu_\mathfrak{P} = \mu_\mathfrak{p}$ and therefore $\mathcal{N}_{K_\mathfrak{P}/F_\mathfrak{p}}(\mu_\mathfrak{P}) = \mu_\mathfrak{p}^3$. On the other hand $\mathcal{N}_{K_\mathfrak{P}/F_\mathfrak{p}}(\mathbb{U}_\mathfrak{p}) = \mathbb{U}^3_\mathfrak{p} = \mathbb{U}_\mathfrak{p}$ and the norm is surjective on principal units. Since $\mathfrak{P} \not\in \mathcal{S}$, $v_\mathfrak{p}(\alpha) = 0$ and $\alpha = \zeta u$ with $\zeta \in \mu_\mathfrak{p}$ and $u \in \mathbb{U}_\mathfrak{p}$. It follows from the above discussion that $\alpha^{1-\tau}$ is a norm in $K_\mathfrak{P}/F_\mathfrak{p}$ if and only if $\zeta^{1-\tau} \in \mu_\mathfrak{p}^3$. Let $\mathfrak{p}_0$ be the prime ideal of $k$ below $\mathfrak{p}$. Assume first that $\mathfrak{p}_0$ is ramified in $F/k$. Then $\mu_\mathfrak{p} \subset k_{\mathfrak{p}_0}$ and $\zeta^{1-\tau} = 1$, thus $\alpha^{1-\tau}$ is a norm in $K_\mathfrak{P}/F_\mathfrak{p}$. Assume now that $\mathfrak{p}_0$ is inert\footnote{By (A3), it cannot be split in $F/k$.} in $F/k$. Denote by $f$ the residual degree of $\mathfrak{p}_0$. The group $\mu_{\mathfrak{p}_0}$ of roots of unity in $k$ of order prime to $p$ has order $p^f-1$. Let $\mathfrak{P}^+ := \mathfrak{P} \cap K^+$. The extension $K^+_{\mathfrak{P}^+}/k_{\mathfrak{p}_0}$ is a tamely ramified cyclic cubic extension. Therefore it is a Kummer extension by \cite[Prop.~II.5.12] {lang:book} and $k_{\mathfrak{p}_0}$ contains the third roots of unity, that is $3$ divides $p^f-1$. Since $\tau$ is the Frobenius element at $\mathfrak{p}_0$ of the extension $F/k_0$, we have $\zeta^{1-\tau} = \zeta^{1-p^f} = (\zeta^{(1-p^f)/3})^3 \in \mu_\mathfrak{p}^3$ and therefore $\alpha^{1-\tau}$ is a norm in $K_\mathfrak{P}/F_\mathfrak{p}$. We have proved that $\alpha^{1-\tau}$ is a norm everywhere locally. It follows by the Hasse Norm Principle that there exists $\beta \in K^\times$ such that $\mathcal{N}_{K/F}(\beta) = \alpha^{1-\tau}$. Let $\mathfrak{P}$ be a prime ideal of $K$ not in $\mathcal{S}$ and, as above, let $\mathfrak{p}$ be the prime ideal of $F$ below $\mathfrak{P}$. Assume first that $\mathfrak{P}$ is ramified or inert in $K/F$, then $v_\mathfrak{P}(\beta) = v_\mathfrak{p}(\alpha^{1-\tau})$ or $\frac{1}{3} v_\mathfrak{p}(\alpha^{1-\tau})$ respectively. In both cases we get  $v_\mathfrak{P}(\beta) = 0$ since $\alpha \in K^\times_{\mathcal{S}}$. If $\mathfrak{P}$ is split in $K/F$ then it must be inert or ramified in $K/K^+$ by (A3). It follows that $v_\mathfrak{P}(\beta^{1-\tau}) = 0$. Therefore $\delta := \beta^{1-\tau} \in K_\mathcal{S}^\times$. We now compute
\begin{equation*}
\mathcal{N}_{K/F}(\delta^{1-\tau}) = \mathcal{N}_{K/F}(\beta)^{(1-\tau)^2} = (\alpha^{1-\tau})^{2(1-\tau)} = (\alpha^{1-\tau})^4. 
\end{equation*}
Thus $\alpha^{1-\tau}$ is the norm of $(\delta/\alpha)^{1-\tau} \in (K_\mathcal{S}^\times)^{1-\tau}$. This concludes the proof of Claim~\ref{claim:H0K=1} and therefore also the proof of Claim~\ref{claim:ukprincipal}. The next claim follows from Claim~\ref{claim:nuk=3uf} and the fact that $\Lambda\mathcal{O} \not= \Lambda$. 

\begin{claim}\label{claim:usurj}
$N_H \cdot U_K^- = U_F^-$.  
\end{claim}

\smallskip

Let $\mathcal{F} := \Fitt_{\mathbb{Z}[H]}(\Cl_K^-)$ be the Fitting ideal of $\Cl_K^-$ as a $\mathbb{Z}[H]$-module. Apply Claim~\ref{claim:mainid} to the ideal $\mathcal{F}$ and call $f$ the element of $\mathcal{F}$ such that $\mathcal{O}/\mathcal{FO} \simeq \mathbb{Z}[H]/f\mathbb{Z}[H]$. Set $\bar\eta' := f \cdot \bar\theta$. We have 
\begin{equation}\label{eq:indetap}
(U_K^- : \mathbb{Z}[H] \cdot \bar\eta') = (\mathbb{Z}[H] : f\mathbb{Z}[H]) = (\mathcal{O}:\mathcal{FO}) = |\Cl_K^-|.
\end{equation}
The last equality follows from Lemma~\ref{lem:p1} and the fact that $\mathcal{FO}$ is the Fitting ideal of $\Cl_K^-$ as an $\mathcal{O}$-module. 

\begin{claim}\label{claim:kappa}
Let $n, m \geq 0$ be two integers. Then there exists $\kappa_{n,m} \in \mathbb{Z}[H]$, unique up to a trivial unit, such that
\begin{equation}\label{eq:propkappa}
\text{Norm}(\kappa_{n,m}) = 2^{n+2m} \quad\text{and}\quad e_0 \kappa_{n,m} = e_0 2^n.
\end{equation}
\end{claim}

\noindent We define
\begin{equation*}
\kappa_{n,m} := 2^n e_0 - (-1)^{n+m} 2^m e_1. 
\end{equation*}
It is clear from its construction that $\kappa_{n,m}$ satisfies~\eqref{eq:propkappa}. One can see also directly that $\kappa_{n,m} \in \mathbb{Z}[H]$ since $2 \equiv -1 \pmod{3}$. It remains to prove the unicity statement. Clearly $e_0 \kappa_{n,m}$ is fixed by construction. On the other hand $e_1\kappa_{n,m}$ is an element of norm $2^{2m}$ in $e_1\mathbb{Z}[H] \simeq \mathbb{Z}[\omega]$. Since $2$ is inert in $\mathbb{Z}[\omega]$, there exists only one element in $\mathbb{Z}[\omega]$ of norm $2^{2m}$ up to units. This concludes the proof of the claim.  

\smallskip

Let $e' \in \mathbb{N}$ be such that $2^{e'} = (\bar{U}_k : \mathcal{N}(\bar {U}_F))$. We now prove the following claim.

\begin{claim}
The integer $e-e'$ is non-negative and even.  
\end{claim}

\noindent We consider the natural map $\bar{U}_k \to \bar{U}_{K^+}/\mathcal{N}(\bar{U}_K)$ that comes from the inclusion $U_k \subset U_{K^+}$. Let $\bar{u} \in \bar{U}_k$ be in the kernel of this map. Thus there exists $\bar{x} \in \bar{U}_K$ such that $\bar{u} = \mathcal{N}(\bar{x})$. Set $\bar{y} := N_H \cdot \bar{x} - \bar{u} \in \bar{U}_F$. We have
\begin{equation*}
\mathcal{N}(\bar{y}) = N_H \cdot \mathcal{N}(\bar{x}) - \mathcal{N}(\bar{u}) = 3 \cdot \bar{u} - 2 \cdot \bar{u} = \bar{u}.  
\end{equation*}
Therefore the kernel of the above map is $\mathcal{N}(\bar{U}_F)$ and there is a well-defined injective group homomorphism from $\bar{U}_k/\mathcal{N}(\bar{U}_F)$ to $\bar{U}_{K^+}/\mathcal{N}(\bar{U}_K)$. This proves that\footnote{This inequality follows also from Claim~\ref{claim:h0iso} below.} $e \geq e'$. The cokernel of this map is
\begin{equation}\label{eq:defcok}
\frac{\bar{U}_{K^+}/\mathcal{N}(\bar{U}_K)}{\bar{U}_k/\mathcal{N}(\bar{U}_F)} \simeq \bar{U}_{K^+}/(\bar{U}_k + \mathcal{N}(\bar{U}_K)). 
\end{equation}
It is a finite $\mathbb{Z}[H]$-module of order $2^{e-e'}$. In particular, the idempotents $e_0$ and $e_1$ act on it. We have $e_0 \cdot \bar{U}_{K^+}/(\bar{U}_k + \mathcal{N}(\bar{U}_K)) = N_H \cdot \bar{U}_{K^+}/(\bar{U}_k + \mathcal{N}(\bar{U}_F)) = 1$. It follows that $\bar{U}_{K^+}/(\bar{U}_k + \mathcal{N}(\bar{U}_K)) = e_1 \cdot \bar{U}_{K^+}/(\bar{U}_k + \mathcal{N}(\bar{U}_K))$ is a $\mathbb{Z}[\omega]$-module. Since $2$ is inert in $\mathbb{Z}[\omega]$, the order of $\bar{U}_{K^+}/(\bar{U}_k + \mathcal{N}(\bar{U}_K))$ is an even power of $2$.  This concludes the proof of the claim. 

\smallskip

Let $\kappa := \kappa_{e'+t_S, (e-e')/2}$. We define
\begin{equation}\label{eq:defeta}
\bar\eta := \pm \kappa \cdot \bar\eta'.
\end{equation}

\noindent The choice of the sign will be done during the proof of the next claim. By Claims~\ref{claim:norm},  \ref{claim:kappa} and \eqref{eq:indetap}, it is direct to see that $\bar\eta$ satisfies (P1). It follows directly from its construction, the fact that $\kappa$ is a $2$-unit, and Lemma~\ref{lem:p2} that it is also a solution of~(P2).

\smallskip

The next step is to prove the following result. 

\begin{claim}\label{claim:chi3} Up to the right choice of sign in \eqref{eq:defeta}, we have
\begin{equation*}
\frac{1}{2} \sum_{g \in G} \chi^3(g) \log |\eta^g|_w = L'_{K/k,S}(0, \chi^3). 
\end{equation*}
\end{claim}

\noindent The $\mathbb{Z}[H]$-module $U_K^-/(\mathbb{Z}[H] \cdot \bar\eta)$ has order not divisible by $3$ since $\bar\eta$ satisfies (P1). Thus it is a $\mathcal{O}$-module and we can split it into two parts corresponding to the two idempotents $e_0$ and $e_1$. On one side, using Claim~\ref{claim:usurj}, we have
\begin{equation*}
e_0 \cdot \big(U_K^-/\mathbb{Z}[H] \cdot \bar\eta\big) = N_H \cdot \big(U_K^-/\mathbb{Z}[H] \cdot \bar\eta\big)\simeq  U_F^-/\mathbb{Z} \cdot \bar\eta_F
\end{equation*}
where $\bar\eta_F := N_H \cdot \bar\eta \in U^-_F$. On the other side, we compute
\begin{equation*}
e_1 \cdot \big(U_K^-/\mathbb{Z}[H] \cdot \bar\eta\big) \simeq  e_1 \big(\mathbb{Z}[H]/\kappa f \mathbb{Z}[H]\big) \simeq \mathbb{Z}[\omega]/2^{(e-e')/2} \mathcal{F}_1
\end{equation*}
where $\mathcal{F}_1$ is the Fitting ideal of $(\Cl_K^-)^{e_1}$ viewed as an $\mathbb{Z}[\omega]$-module. Indeed, we have by construction
\begin{equation*}
e_1 f \mathbb{Z}[H] = e_1 \mathcal{F}\mathcal{O} \simeq \Fitt_{\mathbb{Z}[\omega]}((\Cl_K^-)^{e_1}).
\end{equation*}
Since $\Cl_K^- = (\Cl_K^-)^{e_0}  \oplus (\Cl_K^-)^{e_1}$ and $(\Cl_K^-)^{e_0} \simeq \mathcal{N}_{K/F}(\Cl_K^-)$, we have\begin{equation*}
(\mathbb{Z}[\omega] : \mathcal{F}_1) = |(\Cl_K^-)^{e_1}| = \frac{|\Cl_K^-|}{|\mathcal{N}_{K/F}(\Cl_K^-)|}.
\end{equation*}

\begin{subclaim}\label{claim:nclk}
$\mathcal{N}_{K/F}(\Cl_K^-) = \Cl^-_F$. 
\end{subclaim}

Consider the composition of maps $\Cl_F^- \to \Cl_K^- \to \Cl_F^-$ where the first map is the map induced by the lifting of ideals from $F$ to $K$ and the second map is the norm $\mathcal{N}_{K/F}$. The map constructed in that way is the multiplication by $3$ and therefore, if the order of $\Cl_F^-$ is not divisible by $3$, it is a bijection and the claim is proved. Assume that $3$ divides the order of $\Cl_F^-$. Let $h_E$ denote the class number of a number field $E$. Thus $h_K^- := |\Cl_K^-| = h_K/h_{K^+}$ and $h_F^-  := |\Cl_F^-| = h_F/h_k$. If $K/F$ is ramified at some finite prime then $h_F$ divides $h_K$. As $h_F^-$ divides $h_F$, it follows that $3 \mid h_K$, a contradiction. Assume now that $K/F$ is unramified at finite primes. Therefore $3$ divides $h_F$ and $h_F/3$ divides $h_K$. In the same way, $K^+/k$ is unramified and therefore $3$ divides $h_k$. Since $3 \mid h_F^-$, this implies that $9$ divides $h_F$ and therefore $3$ divides $h_K$, a contradiction. It follows that $3$ does not divide $|\Cl_F^-|$ and the claim is proved. 

\smallskip 

Putting together the claim and the computation that precedes it, we find that 
\begin{equation}\label{eq:eqnf}
(U_F^- : \mathbb{Z} \cdot \bar\eta_F) = \frac{(U_K^- : \mathbb{Z} \cdot \bar\eta)}{2^{e-e'}} \frac{|\Cl_F^-|}{|\Cl_K^-|} = 2^{e'+t_S} |\Cl_F^-|.
\end{equation}

Let $\mathfrak{P}^+ \in S_{K^+}$ and denote by $\mathfrak{p}_0$ the prime ideal of $k$ below $\mathfrak{P}^+$. Then $\mathfrak{P}^+$ is inert in $K/K^+$ if and only if $\mathfrak{p}_0$ is inert in $F/k$. Furthermore,  if $\mathfrak{P}^+$ is inert in $K/K^+$, then it is ramified\footnote{Recall that $S = S(K/k)$.} in $K^+/k$ and it is the only prime ideal in $S_{K^+}$ above $\mathfrak{p}_0$. It follows that the number $t_S$ of prime ideals in $S_{K^+}$ that are inert in $K/K^+$ is equal to the number of prime ideals in $S$ that are inert in $F/k$. Therefore $\bar\eta_F$ satisfy the properties (P1) and (P2) for the extension $F/k$ and the set of primes $S$. As a consequence of Theorem~\ref{th:quadratic}, we see that either $\eta_F$ or $\eta_F^{-1}$ is the Stark unit for the extension $K/F$ and the set of places $S$. By choosing the right sign in \eqref{eq:defeta}, we can assume that $\eta_F$ is the Stark unit. Therefore we have
\begin{equation*}
\frac{1}{2} (\log |\eta_F|_w + \nu(\tau) \log |\eta_F^\tau|_w) = L'_{F/k,S}(0, \nu)
\end{equation*}
where $\nu$ is the non trivial character of $F/k$. It follows from the functorial properties of $L$-functions that $L_{F/k, S}(s, \nu) = L_{K/k, S}(s, \chi^3)$, and from the definition of $\eta_F$ that
\begin{equation*}
\log |\eta_F|_w + \nu(\tau) \log|\eta_F^\tau|_w = \sum_{g \in G} \chi^3(g) \log|\eta^g|_w. 
\end{equation*}
This completes the proof of the claim. 

\smallskip

Now, by Proposition~\ref{prop:prodform}, we know that
\begin{multline*}
\left(\frac{1}{2} \sum_{g \in G} \chi(g) \log |\eta^g|\right) \left(\frac{1}{2} \sum_{g \in G} \chi^3(g) \log |\eta^g|\right) \left(\frac{1}{2} \sum_{g \in G} \chi^5(g) \log |\eta^g|\right) \\
= \pm L'_{K/k,S}(0, \chi) L'_{K/k,S}(0, \chi^3) L'_{K/k,S}(0, \chi^5).
\end{multline*}
We cancel the non zero terms corresponding to $\chi^3$ using Claim~\ref{claim:chi3} and, since $\chi$ and $\chi^5$ are conjugate, we get
\begin{multline*}
\left|\frac{1}{2} \sum_{g \in G} \chi(g) \log |\eta^g|\right|^2 = \left(\frac{1}{2} \sum_{g \in G} \chi(g) \log |\eta^g|\right) \left(\frac{1}{2} \sum_{g \in G} \chi^5(g) \log |\eta^g|\right) \\
= L'_{K/k,S}(0, \chi) L'_{K/k,S}(0, \chi^5) = |L'_{K/k,S}(0, \chi)|^2.
\end{multline*}
Taking square-roots, we get
 \begin{equation*}
\left|\frac{1}{2} \sum_{g \in G} \chi(g) \log |\eta^g|\right| = \left|\frac{1}{2} \sum_{g \in G} \chi^5(g) \log |\eta^g|\right|  = |L'_{K/k,S}(0, \chi)| = |L'_{K/k,S}(0, \chi^5)|.
\end{equation*}
Note that we have directly using \cite[Prop.~I.3.4]{tate:book}
 \begin{multline*}
\frac{1}{2} \sum_{g \in G} \chi_0(g) \log |\eta^g| = \frac{1}{2} \sum_{g \in G} \chi^2(g) \log |\eta^g| = \frac{1}{2} \sum_{g \in G} \chi^4(g) \log |\eta^g| \\ 
= L'_{K/k,S}(0, \chi_0) = L'_{K/k,S}(0, \chi^2) = L'_{K/k,S}(0, \chi^4) = 0. 
\end{multline*}

\smallskip

We now prove that $\bar\eta$ is unique up to multiplication by an element of $H$. Assume that $\bar\eta'$ is another element of $U_K^-$ satisfying (P1), (P2) and such that $N_H \cdot \bar\eta'$ is the Stark unit for the extension $F/k$ and the set of places $S$. Let $u \in \mathbb{Q}[H]$ be such that $\bar\eta' = u \cdot \bar\eta$. By Corollary~\ref{cor:unicity}, $u$ is a $2$-unit. Now, by hypothesis, $\bar\eta_F = N_H \cdot (u \cdot \bar\eta) = u \cdot (N_H \cdot \bar\eta) = u \cdot \bar\eta_F$ and thus $e_0u = e_0$. Write $u_1$ for the element of $\mathbb{Q}(\omega)$ such that $(1, u_1)$ corresponds to $u$ by the isomorphism in \eqref{eq:q6iso}. Since both $\bar\eta$ and $\bar\eta'$ satisfy (P1), we have $\text{Norm}(u) = 1$ and thus $N_{\mathbb{Q}(\omega)/\mathbb{Q}}(u_1) = 1$. But $u_1$ is a 2-unit in $\mathbb{Q}(\omega)$ and there is only prime ideal above $2$ in $\mathbb{Q}(\omega)$. Therefore $u_1$ is in fact a unit and $u \in H$. 

\smallskip

Finally, it remains to prove that $K(\sqrt{\eta})/k$ is an abelian extension. As noted before this is equivalent to prove that $(\gamma - 1) \cdot \bar{\eta} \in 2 \cdot \bar{U}^-_K$ by \cite[Prop.~IV.1.2]{tate:book}. Now $\gamma$ acts on $U_K^-$ as $-\sigma^2$. Thus, by the definition of $\bar{\eta}$ and the isomorphism between $U^-_K$ and $\mathbb{Z}[H]$, this is equivalent to prove that
\begin{equation}\label{eq:abcond}
(\sigma^2 + 1) \kappa f \in 2 \mathbb{Z}[H].  
\end{equation}

\begin{claim}\label{eq:2zh}
Let $x \in \mathbb{Z}[H]$. Then $x \in 2\mathbb{Z}[H]$ if and only if $x e_0 \in 2e_0\mathbb{Z}[H]$ and $xe_1 \in 2e_1\mathbb{Z}[H]$.
\end{claim}

\noindent If $x \in 2\mathbb{Z}[H]$ then clearly $x e_0 \in 2e_0\mathbb{Z}[H]$ and $xe_1 \in 2e_1\mathbb{Z}[H]$. Reciprocally, assume that $x e_0 = 2e_0a_0$ and $x e_1 = 2e_1a_1$ with $a_0, a_1 \in \mathbb{Z}[H]$. Let $a := e_0a_0 + e_1a_1$. We have by construction $2a = x \in \mathbb{Z}[H]$ and $3a = (3e_0)a_0 + (3e_1)a_1 \in \mathbb{Z}[H]$. Therefore $a$ belongs to $\mathbb{Z}[H]$ and the claim is proved.

\smallskip

We prove \eqref{eq:abcond} using the claim. On one hand, we have
\begin{equation*}
 e_0 (\sigma^2 + 1) \kappa f = 2^{e'+t_S+1} e_0 f \in 2e_0\mathbb{Z}[H]. 
\end{equation*}
On the other hand, we have
\begin{equation*}
 e_1 (\sigma^2 + 1) \kappa f = 2^{(e-e')/2} e_1 (\sigma^2 + 1) f.
\end{equation*}
The proof will be complete if we prove that $e - e' > 0$. For that we use the following claim. 

\begin{claim}\label{claim:h0iso}
$|\hat{H}^0(T, U_K/U_F)| = 2^{e-e'}.$ 
\end{claim}

Let $U_K^\circ$ be the subgroup of elements $u \in U_K$ such that $u^{1 - \tau} \in U_F$. We have
\begin{align*}
\hat{H}^0(T, U_K/U_F) & = \frac{(U_K/U_F)^T}{\mathcal{N}(U_K/U_F)} = \frac{U_K^\circ/U_F}{\mathcal{N}(U_K)/\mathcal{N}(U_K) \cap U_F} \\
& \simeq \frac{U_K^\circ/U_F}{(\mathcal{N}(U_K) \, U_F)/U_F} \simeq U_K^\circ/(\mathcal{N}(U_K) \, U_F) \\
& \simeq \bar{U}_K^\circ/(\mathcal{N}(\bar{U}_K) + \bar{U}_F). 
\end{align*}
By \eqref{eq:defcok}, it is enough to prove the following isomorphism
\begin{equation}\label{eq:prove}
\bar{U}_{K^+}/(\mathcal{N}(\bar{U}_K) + \bar{U}_k) \simeq \bar{U}_K^\circ/(\mathcal{N}(\bar{U}_K) + \bar{U}_F).
\end{equation}
Since $\bar{U}_{K^+} \cap (\mathcal{N}(\bar{U}_K) + \bar{U}_F) = \mathcal{N}(\bar{U}_K) + \bar{U}_k$, there is a natural injection of the LHS of \eqref{eq:prove} in the RHS induced by the inclusion $\bar{U}_{K^+} \subset \bar{U}_K^\circ$. We prove now that this map is surjective. Let $\bar{u} \in \bar{U}_K^\circ$. Thus $\bar{x} := (1 - \tau) \cdot \bar{u} \in \bar{U}_F$. Note that $(1 - \tau) \cdot \bar{x} = 2 \cdot \bar{x}$. Define $\bar{y} := N_H \cdot \bar{u} - \bar{x} \in \bar{U}_F$ and $\bar{z} := \bar{u} - \bar{y}$.  We have
\begin{equation*}
(1 - \tau) \cdot \bar{z} = (1 - \tau) \cdot \bar{u} - (1 - \tau) N_H \cdot \bar{u} + (1 - \tau) \cdot \bar{x} = \bar{x} - N_H \cdot \bar{x} + 2 \cdot \bar{x} = 0. 
\end{equation*}
Thus $\bar{z} \in \bar{U}_{K^+}$. This proves that $\bar{u} = \bar{z} + \bar{y} \in \bar{U}_{K^+} + \bar{U}_F$. Equation \eqref{eq:prove} follows and the proof of the claim is finished. 

\smallskip

Now by the multiplicativity of the Herbrand quotient and Lemma~\ref{lem:qu}, we find that
\begin{equation}\label{eq:qukf}
Q(T, U_K/U_F) = \frac{Q(T, U_K)}{Q(T, U_F)} =  2^{2d-2}. 
\end{equation}
Therefore $e - e' \geq 2d - 2 \geq 2$. This concludes the proof that $K(\sqrt{\eta})$ is abelian over $k$ and the proof of the theorem. 
\end{proof}

\begin{corollary}\label{cor:square6}
Under the hypothesis of the theorem and assuming that the Stark unit exists, then it is a square in $K$ if and only if the Stark unit for the extension $F/k$ and the set $S$ is a square and $(e - e')/2 + c - c' \geq 1$ where $c$ is the $2$-valuation of $|\Cl_K^-|$, $c'$ is the $2$-valuation of $|\Cl_F^-|$ and $(\bar{U}_k : \mathcal{N}(\bar U_F)) = 2^{e'}$. In particular, if $d \geq 4$ then it is always a square and, in fact, it is a $2^{d-3}$-th power. It is also a square if $d = 3$ and the extension $K/k$ is ramified at some finite prime.
\end{corollary}

\begin{proof}
We use the notations and results of the proof of the theorem. By the unicity statement, the Stark unit, if it exists, is equal to $\eta$ or one of this conjugate over $F$. In particular, the Stark unit is a $2^r$-th power in $K$ if and only if $\bar\eta \in 2^r \cdot U_K^-$. By Claim~\ref{eq:2zh} and the construction of $\bar\eta$, this is equivalent to
\begin{equation*}
2^{e'+t_S} e_0 f \in 2^r e_0\mathbb{Z} \quad\text{and}\quad 2^{(e-e')/2} e_1 f \in 2^r e_1\mathbb{Z}[H].  
\end{equation*}
Now $e_0 f \mathbb{Z} = e_0 |\Cl_F^-| \mathbb{Z}$ by the definition of $f$, Claim~\ref{claim:nclk} and the discussion that precedes it. Thus the first condition is equivalent to $e'+t_S+c' \geq r$. For $r = 1$, this is equivalent to the fact that the Stark unit for $F/k$ and the set $S$ is a square by Theorem~\ref{th:quadratic} and the discussion that follows \eqref{eq:eqnf} on the number of primes in $S$ that are inert in $F/k$. For the second condition, recall that $e_1f\mathbb{Z}[H] \simeq \Fitt_{\mathbb{Z}[\omega]}((\Cl_K^-)^{e_1})$ and therefore $e_1f \in 2^v \mathbb{Z}[H]$ where $v$ is the $2$-valuation of the index $(\mathbb{Z}[\omega]:\Fitt_{\mathbb{Z}[\omega]}((\Cl_K^-)^{e_1}))$. By Claim~\ref{claim:nclk} and the computation before it, this index is equal to $|\Cl_K^-|/|\Cl_F^-|$. Therefore the second condition is equivalent to $(e-e')/2 + c-c' \geq r$. This proves the first assertion: the Stark unit for $K/k$ and $S = S(K/k)$ is a square if and only if the Stark unit for the extension $F/k$ and the set $S$ is a square and $(e - e')/2 + c - c' \geq 1$. For the second assertion, we have $e' \geq d-3$ by \eqref{eq:bounde} and $(e-e')/2 \geq d-1$ by Claim~\ref{claim:h0iso} and \eqref{eq:qukf}. Thus $\bar\eta \in 2^{d-3} \cdot U_K^-$ for $d \geq 4$ and we have that the Stark unit is a $2^{d-3}$-th power if $d \geq 4$. Finally, for $d = 3$, the condition $2^{(e-e')/2} e_1 f \in 2 e_1\mathbb{Z}[H]$ is always satisfied. Assume that the extension $K/k$ is ramified at some finite prime. If $F/k$ is also ramified at some finite prime then the Stark unit for the extension $F/k$ and the set $S$ is a square by Theorem~\ref{th:quadratic}. If $F/k$ is unramified at finite primes then any prime ideal that ramifies in $K/k$ is inert in $F/k$ by (A3). Therefore $t_S \geq 1$ and the Stark unit for the extension $F/k$ and the set $S$ is a also square by Theorem~\ref{th:quadratic}. It follows that the Stark unit for $K/k$ is a square by the first part. This concludes the proof. 
\end{proof}

Note that the condition in the case $d = 3$ is sharp. Indeed let $k := \mathbb{Q}(\alpha)$, where $\alpha^3 + \alpha^2 - 9\alpha - 8 = 0$, be the smallest totally real cubic field of class number $3$. Let $v_1, v_2, v_3$ be the three infinite places of $k$ with $v_1(\alpha) \approx -3.0791$, $v_2(\alpha) \approx -0.8785$ and $v_3(\alpha) \approx 2.9576$. Let $K$ be the ray class field of $k$ of modulus $\mathbb{Z}_k v_2 v_3$. The extension $K/k$ is cyclic of order $3$, satisfies (A1), (A2) and (A3) with $S := (S/k)$, and is unramified at finite places. One can check that, if it exists, the Stark unit is not a square in $K$. 

\bibliographystyle{plain}
\bibliography{refs}

\end{document}